\documentclass[a4paper,oneside,10pt,reqno]{amsart} 
\usepackage[a4paper]{geometry}
\usepackage[dvipsnames]{xcolor}

\usepackage{pifont}

 \usepackage[applemac]{inputenc}
 
\usepackage{multicol} 

\usepackage{mathrsfs}
\usepackage{mathabx}

\usepackage[all]{xy}

\usepackage{hyperref}
\usepackage{graphicx}
\usepackage{amssymb}

\usepackage{enumerate}

\usepackage{tikz}

\newtheorem{theorem}{Theorem}

\newtheorem{proposition}{Proposition}[section]
\newtheorem{lemma}[proposition]{Lemma}
\newtheorem{corollary}[proposition]{Corollary}
\newtheorem{theoremb}[proposition]{Theorem}

\theoremstyle{definition}
\newtheorem{definition}[proposition]{Definition}
\newtheorem{example}[proposition]{Example}
\newtheorem{examples}[proposition]{Examples}

\newtheorem{remark}[proposition]{Remark}

\newcommand{\bl}{\begin{lemma}}
\newcommand{\bp}{\begin{proposition}}
\newcommand{\bt}{\begin{theorem}}
\newcommand{\bc}{\begin{corollaire}}
\newcommand{\be}{\begin{equation}}
\newcommand{\bee}{\begin{equation*}}
\newcommand{\bd}{\begin{definition}}
\newcommand{\bdp}{\begin{definitionproposition}}
\newcommand{\bex}{\begin{example}}
\newcommand{\br}{\begin{remarque}}
\newcommand{\bpr}{\begin{proof}}

\newcommand{\el}{\end{lemma}}
\newcommand{\ep}{\end{proposition}}
\newcommand{\et}{\end{theorem}}
\newcommand{\ec}{\end{corollaire}}
\newcommand{\ee}{\end{equation}}
\newcommand{\eee}{\end{equation*}}
\newcommand{\ed}{\end{definition}}
\newcommand{\edp}{\end{definitionproposition}}
\newcommand{\eex}{\end{example}}
\newcommand{\er}{\end{remarque}}
\newcommand{\epr}{\end{proof}}

\newcommand{\secref}[1]{Section~\ref{#1}}

\newcommand{\thmref}[1]{Theorem~\ref{#1}}
\newcommand{\propref}[1]{Proposition~\ref{#1}}
\newcommand{\lemref}[1]{Lemma~\ref{#1}}
\newcommand{\corref}[1]{Corollary~\ref{#1}}

\newcommand{\remref}[1]{Remark~\ref{#1}}
\newcommand{\exemref}[1]{Example~\ref{#1}}
\newcommand{\exemsref}[1]{Examples~\ref{#1}}

\newcommand{\defref}[1]{Definition~\ref{#1}}

\def\ov{\overline}

\def\menos{\setminus}

\renewcommand{\int}[1]{{\rm int} (#1)}

\def\tc{{\mathtt c}}
\def\tv{{\mathtt v}}

\def\tu{{\mathtt u}}

\def\T{\mathcal G}

\def\T{\mathcal T}

\def\I{\mathcal I}

\def\SS{\mathcal S}

\def\TT{\mathcal T}

\def\N {\mathbb N}

\def\R {\mathbb R}

\def\Z {\mathbb Z}
\renewcommand\S {\mathbb S}
\renewcommand\L {\mathcal L}

\def\cI{{\mathcal I}}
\def\cL{{\mathcal L}}
\def\cR{{\mathcal R}}
\def\cS{{\mathcal S}}
\def\cT{{\mathcal T}}
\def\cV{{\mathcal V}}
\def\cX{{\mathcal X}}

\def\Id{{\rm id\,}}

\def\codim{{\rm codim}}
\def\pr{{\rm pr}}

\def\depth{{\rm depth \,}}

\def\rc{{\mathring{\tc}}}

\def\Sing{{\rm Sing}}

\def\crG{{\mathscr G}}

\def\gE{{\mathfrak E}}
\def\gU{{\mathfrak U}}

\makeatletter
\@namedef{subjclassname@2020}{%
  \textup{2020} Mathematics Subject Classification}
\makeatother

\title{Intersection homotopy, refinements and coarsenings}

\date{\today}

\author{Martintxo Saralegi-Aranguren}
\address{Laboratoire de Math{\'e}matiques de Lens\\  
      EA 2462 \\
      Universit\'e d'Artois\\
         SP18, rue Jean Souvraz\\
          62307 Lens Cedex\\
         France}
\email{martin.saraleguiaranguren@univ-artois.fr}

\author{Daniel Tanr\'e}
\address{D\'epartement de Math{\'e}matiques\\
         UMR-CNRS 8524 \\
         Universit\'e de Lille\\
         59655 Villeneuve d'Ascq Cedex\\
         France}
\email{Daniel.Tanre@univ-lille.fr}

\thanks{The second author was partially supported by 
the ANR-11-LABX-0007-01  ``CEMPI''}



\makeindex   

\begin{document}

\begin{abstract} 
In previous works, we studied the intersection homotopy groups 
associated to a Goresky and MacPherson perversity $\ov p$
and a filtered space $X$.
They are defined as the homotopy groups of simplicial sets  introduced by P. Gajer. 
We particularized to locally conical  spaces  of Siebenmann (called CS sets) and established a topological invariance 
for them when the regular part remains unchanged.

Here, we consider coarsenings, made of two structures of CS sets on the same topological space, 
the strata of one being a union of strata of the other.
We endow them with a general perversity and its pushforward, where  the adjective ``general'' 
means that the perversities are defined on the poset of the strata and not only according to their codimension.
If the perversity verifies a growing property analogous to that of the original perversities of Goresky and MacPherson, 
we also find an invariance theorem for the intersection homotopy groups of a coarsening, 
under the above restriction on the regular parts. 
An invariance is shown too in some cases where  singular strata become regular in the coarsening, for
Thom-Mather spaces.
\end{abstract}

\maketitle

\tableofcontents

\newpage
This work is concerned with  intersection homotopy groups of filtered spaces, $(X,\cS)$,
endowed with a perversity $\ov{p}$. They are
defined as the homotopy groups of  the simplicial set
$\crG_{\ov p}(X,\cS)$ introduced by Gajer (\cite{Gajer1,Gajer2}); i.e.,
$\pi_{\star}^{\ov p}(X,\cS,x_{0})=\pi_{\star}(\crG_{\ov p}(X,\cS),x_{0})
$.
For instance, $(X,\cS)$ is said $\ov p$-connected if $\crG_{\ov p}(X,\cS)$ is connected.

\medskip
In previous works (\cite{CST10,CST8,CST9}),  
we studied the relations between these  groups
and the intersection homology groups of Goresky and MacPherson (\cite{GM1}),
establishing a Hurewicz theorem and a Van Kampen theorem.
We refer the interested reader  to \cite{CST9} for more details.
Here, we focus on the topological invariance of the $\ov{p}$-intersection homotopy groups
defined from  a perversity $\ov{p}$ and a  topological space endowed with different filtrations.

\medskip
 Let's start by clarifying the situation. 
 In \cite{GM1,GM2}, topological invariance  means the existence of isomorphisms between the intersection
homology groups of a given pseudomanifold for any 
 fixed  GM-perversity (see \defref{def:GMperversity})
independently of the chosen filtration.
In  \cite{CST9}, we were looking for conditions that guarantee the invariance of intersection homotopy groups
within the framework of  GM-perversities
and locally conical filtered spaces with a structure of manifold on each stratum. 
These spaces are the CS sets of Siebenmann (\cite{MR0319207}) recalled in \defref{def:csset}.
From King's work (\cite{MR800845}) on the topological invariance of  intersection homology without sheaf,
 there is a  preferred topological pattern. King's sketch of proof is as follows:
 given a CS set $(X,\cS)$, he constructs an explicit intrinsic CS set $(X,\cX^\star)$ which is
 the coarsest  filtration on $X$ as a CS set.
  The proof of topological invariance reduces to the existence of an isomorphism
between the $\ov{p}$-intersection homology groups of $(X,\cS)$ and of $(X,\cX^\star)$.
In  \cite[Theorem A]{CST9}, 
we adapted this idea to the invariance of $\ov{p}$-intersection homotopy groups. The example
of a double suspension on a Poincar\'e sphere shows that such invariance cannot be  true in full generality.
We established it  for the CS sets  $(X,\cS)$ having the same regular part than $(X,\cX^\star)$.

 \medskip
 In the present work, we come back to the topological invariance of intersection homotopy groups
 and establish it with weaker requirements.
 First, we consider more general perversities, $\ov{p}\colon \cS\to \ov\Z=\Z\cup\{-\infty,\infty\}$,
 introduced by MacPherson (\cite{RobertSF}). Defined on the family $\cS$ of strata of a filtered
 space $X$, they take values in $\ov\Z$.
 (We already  addressed this situation in the context of intersection homology in \cite{CST3}.)
 Since these perversities are defined on the set of strata, they depend on the chosen filtration and we could conclude 
 that there is no hope of topological invariance with this generality. 
 In fact, we must not forget that, for a topological invariance, only one GM-perversity appears. 
 Here we  have two filtrations, therefore a priori two perversities, and we have to
 impose a link between them,  that will bring the invariance.  Let's explain this.
 
\medskip
We handle with triples $(X,\cS,\cT)$ formed of a topological space $X$, endowed with two structures of CS sets, 
$(X,\cS)$ and $(X,\cT)$, 
such that any stratum of $\cT$ is a union of strata of $\cS$. 
The condition required between the strata of $\cS$ and  $\cT$ is equivalent to asking that the identity map defines a
stratified map, $\iota\colon (X,\cS)\to (X,\cT)$. We call $(X,\cS,\cT)$  a \emph{CS-coarsening.}
Let us call \emph{exceptional} a singular stratum of $\cS$ which is in the regular part of $\cT$.
The absence of exceptional strata means that $(X,\cS)$ and $(X,\cT)$ have the same regular part.

\medskip
For the starting point with perversities, we have two choices: 
a perversity $\ov{p}$ defined on $\cS$ or a perversity $\ov{q}$ on $\cT$.
As $\iota$ is a stratified map, a  perversity $\ov{p}$ on $\cS$  
generates a pushforward perversity $\iota_{\star}\ov{p}$ 
on $\cT$ and a perversity $\ov{q}$ on $\cT$ produces a  pullback perversity $\iota^\star\ov{q}$ on $\cS$, see \defref{def:pushback}. In conclusion, we are faced with the existence of two isomorphisms, 
one connecting $\ov{p}$ and $\iota_{\star}\ov{p}$  and the other associating $\iota^\star\ov{q}$ and $\ov{q}$.
Before stating them, we still need to specify the required properties on perversities.
GM-perversities are defined to satisfy, in particular,  the following key  property,
\begin{equation}\label{equa:GMperversity}
\ov{p}(i)\leq \ov{p}(i+1)\leq \ov{p}(i)+1.
\end{equation}
We adapt these inequalities to our more general context  and obtain the notion of $K$-perversity
on a CS-coarsening, see \defref{def:Kper}. 
We begin with the relation between a perversity $\ov p$ on $\cS$ 
and its pushforward perversity $\iota_{\star}\ov p$ on $\cT$.
The first result extends \cite[Theorem A]{CST9} to general perversities.
(Recall that $\ov t$ is the top perversity, defined by
$\ov t(S)=\codim\,S-2$, if $S$ is singular.)

\begin{theorem}\label{thm:coarseningnoexcep}
Let $(X,\cS,\cT)$ be a CS-coarsening
without exceptional strata and  $\ov{p}$ be a  $K$-perversity on $(X,\cS,\cT)$
of pushforward perversity $\iota_{\star}\ov p$ on $\cT$, such that $\ov p\leq \ov t$.
Then, for any regular point $x_{0}$, the identity map induces an isomorphism
$$\pi_{\star}^{\ov{p}}(X,\cS,x_{0})\cong \pi_{\star}^{\iota_{\star}\ov{p}}(X,\cT,x_{0}).$$
\end{theorem}

In \cite{CST9}, we have shown the topological invariance of the intersection homotopy groups for 
GM-perversities and CS sets without exceptional strata. \thmref{thm:coarseningnoexcep}  implies this result, in the 
more general  paradigm of the perversities used by H. King (\cite{MR800845}).
In the next result, we also go beyond the  context of \cite{CST9}, by admitting  some exceptional strata.
 With the example of a double suspension on a Poincaré sphere 
 (\cite[Example 5.3]{CST9} or its recall in \exemref{exam:henri}), we know that an invariance
 cannot exist  without additional hypotheses on the exceptional strata.
 We succeed for Thom-Mather spaces with  a ``local hypothesis'' on  links.
In fact, we only need a weaker notion of Thom-Mather space, requiring only  the existence of
 neighborhoods of strata, without control data, see \defref{def:petitThomMather}.

\begin{theorem}\label{thm:coarseningwithexcep}
Let $(X,\cS,\cT)$ be a CS-coarsening such that $(X,\cS)$  is a normal connected Thom Mather space,
with a finite number of strata.
Let  $\ov{p}$ be a  $K$-perversity  on $(X,\cS,\cT)$,
of pushforward perversity $\iota_{\star}\ov p$ on $\cT$, such that $\ov p\leq \ov t$.
We suppose that the links, $L$, of all exceptional strata verify
$\pi_{1}^{\ov p}(\rc L,\rc\cS)=\{1\}$.
Then, for any regular point $x_{0}$, the identity map induces an isomorphism
$$\pi_{\star}^{\ov{p}}(X,\cS,x_{0})\cong \pi_{\star}^{\iota_{\star}\ov{p}}(X,\cT,x_{0}).$$
\end{theorem}

Let $(X,\cS,\cT)$ be a CS-coarsening without exceptional strata of codimension~1.
If $\ov{q}$ is a perversity on $(X,\cT)$, we observe 
(\lemref{Dback})
that the pullback perversity $\iota^\star\ov q$  
is  a $K$-perversity on $(X,\cS,\cT)$ and that we have $\iota_{\star}\iota^\star\ov{q}=\ov{q}$.
Thus the two previous results, applied to $\ov{p}:=\iota^\star\ov{q}$,  
 imply  isomorphisms induced by the identity map,
$$
\pi_{\star}^{\iota^\star\ov{q}}(X,\cS,x_{0})\cong \pi_{\star}^{\ov{q}}(X,\cT,x_{0}),
$$
see Theorems \ref{thm:raffinement} and \ref{thm:raffinementThom} for explicit statements.

\medskip
In \propref{prop:GM}, we  refind the topological invariance of  \cite[Theorem 6.1]{CST9}
under the slight more general form of perversities introduced by H.~King in \cite{MR800845}.
 In this context, the intersection homotopy groups do not depend on the chosen stratification,
 if the regular part remains unchanged.

\medskip
Finally, let’s add a word about the method used in the proofs.
We follow the scheme introduced and developed in \cite{MR4170473,MR4742271}:
any CS-coarsening $(X,\cS,\cT)$  can be decomposed in a series of CS-coarsenings 
in which, roughly, only two strata of $\cS$ are merged to give a stratum of $\cT$. The corresponding conical charts
can be of three types which are described in  the {\it op. cit.} and recalled in \secref{sec:coarsening}. 
They are the keys in the proofs of the main theorems.

\medskip
Let's specify some notation and convention.
Here a topological space, or more simply a space, means a compactly generated topological space and 
our definition of compact space includes the Hausdorff property. 
For a   space $X$, we denote by $\tc X$ the quotient $\tc X:= X \times [0,1]/ (X \times \{ 0\})$.
The \emph{open cone} $\rc X$ is the quotient $\rc X:= X \times [0,1[/ (X \times \{ 0\})$.
 A point of the cone is denoted by $[x,t]$ and the apex is $\tv =[-,0]$.
The symbol $\ast$ is reserved for join operations between filtered spaces.
Indices or superscripts are designated by 
$\star$.

\smallskip
Let $m\in \N$. 
The  sphere of dimension $m$ is denoted $S^m$ and $\S^m$ denotes a homological sphere of dimension $m$.
All stratified maps defined by the identity map are  designed by $\iota$ and their induced homomorphisms,
in homology or homotopy, by $\iota_{\star}$. 

\smallskip
{\bf Acknowledgements.} We are grateful to the referee for her/his comments and suggestions 
which helped us to improve the manuscript.

\section{Filtered spaces and CS sets}
\begin{quote}
The singular spaces of  this section are general filtered spaces and the  CS sets  
introduced by Siebenmann (\cite{MR0319207}).
\end{quote}

\begin{definition}\label{def:filteredspace}
A \emph{filtered space of formal dimension $n$} is a non-empty space $X$ 
with a filtration,
$$
X_{-1}=\emptyset\subseteq X_{0} \subseteq X_{1}\subseteq \dots \subseteq X_{n-2} \subseteq X_{n-1} \subsetneq X_{n} =X,
$$
by closed subsets. 
The subspaces $X_{i}$ are called \emph{skeleta} of the filtration and the 
non-empty components of $X^{i}=X_{i}\menos X_{i-1}$ are the \emph{strata}  of $X$. 
The set of strata, also called the \emph{stratification} of $X$, is denoted by $\cS_{X}$ (or $\cS$ if there is no ambiguity).
The subspace $\Sigma_{\cS}=X_{n-1}$ (or $\Sigma$ if there is no ambiguity) is called the \emph{singular subspace.}
Its complementary, $X\menos \Sigma_{\cS}$, is the \emph{regular subspace} and the
strata in $X\menos \Sigma_{\cS}$ are called \emph{regular.} 
The formal dimension of a stratum $S\subset X^{i}$ is $\dim S=i$; its formal codimension is
$\codim\, S=n-i$. (This notion is not necessarily related to a topological notion of dimension.)
As we are working with different filtrations on the same topological space, we  often use the expression
``the filtered space $(X,\cS)$'' which means that $X$ is a filtered space and $\cS$ is its set of strata.
\end{definition}

The simplest example of filtered space is a non-empty space with the stratification by its connected components; we denote it by  $(X,\cI_{X})$. 

\begin{examples}\label{exam:stratifications}
Let $(X,\cS)$ be a filtered space of dimension $n$.
An open subset $U \subset X$ is a filtered space for the \emph{induced filtration}, $\cS_{U}$, given by
$U_i = U \cap X_{i}$. 
For the sake of simplification, if there is no ambiguity, we keep the notation $\cS$ for the stratification on $U$
and denote by $(U,\cS)$ this induced filtered space.
The product $Y\times X$ with  a  topological space $Y$  is a filtered space for the \emph{product filtration} defined by
$\left(Y \times X\right) _i = Y \times X_{i}$. 
The stratification is denoted by $\cI_{Y}\times \cS$; 
its elements are products of a connected component of $Y$ with a stratum of $\cS$.
 If $X$ is compact, the open cone 
$\rc X = X \times [0,1[ \big/ X \times \{ 0 \}$ is endowed with the 
\emph{cone filtration} defined by
$\left(\rc X\right) _i =\rc X_{i-1}$,  $0\leq i\leq n+1$. 
(By convention, $\rc \,\emptyset=\{ \tv \}$.)
The stratification, denoted by $\rc \cS$, is constituted of the apex $\{\tv=[-,0]\}$
and the products $]0,1[\times S$ with a stratum $S\in\cS$.
\end{examples} 

\begin{definition}\label{def:startifiedmap}
A \emph{stratified map} $f\colon X\to Y$ between two filtered spaces is a continuous map such that,
for each stratum $S$ of $X$, there exists a stratum $S^f$ of $Y$ with $f(S)\subset S^f$
and $\codim\, S^f\leq \codim\, S$.
A \emph{stratified homeomorphism} is a homeomorphism such that $f$ and $f^{-1}$ are stratified maps.
\end{definition}

A stratified map $f$ which is a homeomorphism but for which the reverse application 
$f^{-1}$ is not stratified is called a \emph{homeomorphism and a stratified map.}  
It is fundamental not to confuse it with a stratified homeomorphism.

\begin{definition}\label{def:csset}
A \emph{CS set} of dimension $n$ is a Hausdorff filtered space $(X,\cS)$ of dimension $n$,
whose $i$-dimensional strata are $i$-dimensional topological manifolds for each~$i$, and
such that for each point $x \in X_i \backslash X_{i-1}$, $i\neq n$, there exist
 an open neighborhood $V$ of  $x$ in $X$, endowed with the induced filtration,
an open neighborhood $U$ of $x$ in $X_i\backslash X_{i-1}$, 
a compact filtered space, $(L,\cL)$, of (formal) dimension $n-i-1$,
and a stratified homeomorphism, involving the stratifications recalled in \exemsref{exam:stratifications}:
$$\varphi \colon (U \times \rc L,\cI_{U}\times \rc\cL)\to (V,\cS).$$
 If $\tv$ is the apex of the cone $\rc L$, the homeomorphism
 $\varphi$ is also required to verify $\varphi(u,\tv)=u$, for each $u\in U$.
The pair  $(V,\varphi)$ is a   \emph{conical chart} of $x$
 and the filtered space $L$ is a \emph{link} of $x$.
 The CS set $X$ is called  \emph{normal} if its links are connected.
\end{definition}
Let us note that  a CS set has a topological dimension which coincides with the formal one.
We collect below some well-known basic properties of a CS set $(X,\cS)$, with a proof reference.
\begin{itemize}
\item Any stratum of $X$ is a connected manifold and its closure is a union of strata of lower dimension
(\cite[Lemma 2.3.7]{LibroGreg}).
\item Since the links are non-empty sets, the open subset $X\menos \Sigma$ is dense.
So, for each link $L$, the regular part $L\menos \Sigma_{L}$ is dense in $L$.
\item The set $\cS$ of strata  is a poset (\cite[Proposition A.22]{CST1})
for the relation
$S_{0}\preceq S_{1}$ if $S_{0}\subseteq \ov{S}_{1}$. 
(If $S_{0}\neq S_{1}$, we note $S_{0}\prec S_{1}$.)
 The closed strata are exactly the minimal strata. The open strata are the maximal strata so they coincide with
 the $n$-dimensional strata. 
 The relation $\prec$ is compatible with the dimension, in the sense that 
 $S\prec S'$ implies $\dim S< \dim S'$.
  The \emph{depth of $(X,\cS)$}    is the greatest integer $m$ for which it exists a chain of strata,
 $S_{0}\prec S_{1}\prec \cdots\prec S_{m} $.
  We denote it by $\depth X=m$.
 \item  For any singular stratum $S$, there is a regular stratum $Q$ such that  $S\prec Q$.
 \item Strata are locally closed 
and form a \emph{locally finite family} (\cite[Lemma 2.3.8]{LibroGreg}) which means that  every point has a neighborhood that intersects only a finite number of strata.
 \item  The links of a stratum are not uniquely defined but if  one of them is connected, so all of them are too
(\cite[Remark 2.6.2]{LibroGreg}).
 \item If $X$ is a manifold, then any link is a homology sphere. So, if $X$ has no strata of codimension 1, it is normal
 (\cite[Remark 2.5]{CST9}).
 \end{itemize}

\section{Coarsenings}\label{sec:coarsening}
\begin{quote}
In this section, we specify the notion of CS-coarsenings   and  their main properties.
We also develop  the 3 types of strata encountered in this situation, with an explicit description of the associated local structures.
\end{quote}

\begin{definition}\label{def:coarsening}
A \emph{coarsening} is a triple
$(X,\cS,\cT)$ formed  of two filtered spaces, 
$(X,\cS)$ and $(X,\cT)$, on the same topological space $X$
such that any stratum in $\cT$ is a union of strata of $\cS$. 
We also say that $(X,\cS)$ is a \emph{refinement} of $(X,\cT)$.
If $S\in\cS$, we denote by $S^\iota\in\cT$ the unique stratum of $\cT$ which contains $S$.

If both filtered spaces of a coarsening are CS sets, we speak of a \emph{CS-coarsening.}
A \emph{pointed CS-coarsening} is a quadruplet $(X,\cS,\cT,x_{0})$ formed of a CS-coarsening and a regular point 
$x_{0}\in X\menos \Sigma_{\cS}\subset X\menos\Sigma_{\cT}$.
In the rest of this text, we only consider CS-coarsenings.

Any CS set $(X,\cS)$
has an \emph{intrinsic CS-coarsening} $(X,\cS,\cX^\star)$, where $\cX^{\star}$ is the coarsest  filtration on $X$
as a CS set, see \cite{MR800845} or \cite[Proposition 2.10.5]{LibroGreg}.

\end{definition}

In general a coarsening of a CS set is not a CS set (\cite[Remark 4.3]{MR4170473}). 
In fact, there are even coarsenings,
$(X,\cS,\cR)$ and $(X,\cR,\cT)$, such that $(X,\cS)$ and $(X,\cT)$ are CS sets and $(X,\cR)$
is not. 
We present an example of this situation.

 \begin{example}\label{exam:encadrement}
We give a chain of coarsenings 
 $$
 (X,\cS) \xrightarrow{\iota_1} (X,\cR)  \xrightarrow{\iota_2}  (X,\cT)
 $$
whose  extremities are CS sets and the middle is merely a filtered space.
Let $T^2 \subset \R^3$ be the torus. 
Its cone, $\rc  T^2$ is a closed subset of  $\R^3 \times ]-1,1[$. 
We define $X =\R^3 \times ]-1,1[ \times \R$, choose
two different points $\{P,Q\}$ of $ T^2$ and denote by $\tv$ the apex of the cone.
We define three filtrations on the space  $X$ as follows.
\begin{itemize}
\item $\cS    : X_0 = \{ (\tv , 0) \} \subset X_1 = \rc \{P,Q\} \times \{0\} \subset X_3 = \rc  T^2 \times \{0\} \subset X_4=X$.
 \item $\cR   :  X_1 = \rc \{P,Q\} \times \{0\} \subset X_3 = \rc  T^2 \times \{0\} \subset X_4=X$.
 \item $\cT    :  X_4=X$.
 \end{itemize}
 Notice that $\cS$ is the filtration $\cR_{(\tv,0)}$ as defined in \exemref{exam:joint}.
 The figure below illustrates the subspace $X_3$ with $\cS$ and $\cR$ stratifications.
 \begin{figure}[h]
\includegraphics[width=4cm]{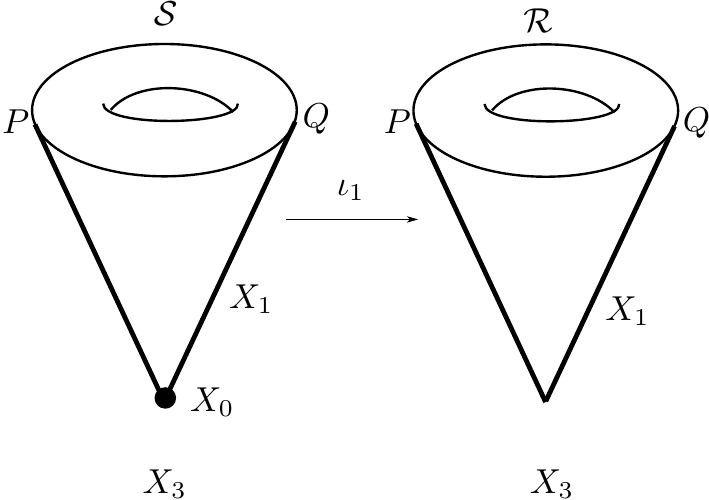}
\end{figure}
 
 The space $X$ is homeomorphic to $\R^5\cong\rc S^4$.
The  stratification $\cS$ of $X$ is the conic filtration where $S^4$ has for only singular stratum 
the torus $T^2$ viewed as a subspace of $S^3\subset S^4$. It is therefore a CS set.

If $(X,\cR)$ were a CS set, then  $X_1\subset X_3$ would be a CS set 
(this follows directly from the definition of a CS set). 
In that case, any two points in $X_1$ would have the same homological link (\cite[Lemma 6.3.23]{LibroGreg}). 
This is not the case, since the link of $\tv \in X_1$ is $ T^2$, while the link of  $P \in X_1$ is $S^2$.
 \end{example}


Let $(X,\cS,\cT)$ be a CS-coarsening.
The condition required between the strata of $\cS$ and $\cT$ is equivalent to asking that the identity map defines a
stratified map, $\iota\colon (X,\cS)\to (X,\cT)$. 
A regular stratum of $\cS$ is necessarily included in a regular stratum of $\cT$. 
For a singular stratum, the situation is different: 
a singular stratum can be included in either a regular stratum or a singular stratum of $\cT$.
Let us begin with the first case.

\begin{definition}\label{def:exceptional}
Let $(X,\cS,\cT)$ be a CS-coarsening. A \emph{singular} stratum $S \in \cS$ is \emph{exceptional} in $(X,\cS,\cT)$ if  
$S^\iota$ is a regular stratum in $\cT$.
\end{definition}

If $S$ is exceptional,  we have $\dim S<\dim S^\iota$.
More generally, the fact that the coarsening keeps or not the dimension of a given stratum is a fundamental distinction. 
So we introduce the following definitions.
  
  \begin{definition}\label{def:source&virtual}
  Let $(X,\cS,\cT)$ be a CS-coarsening.
  A stratum $S\in \cS$ is a \emph{source stratum} in $(X,\cS,\cT)$ if $\dim S = \dim S^\iota$.
 A non-exceptional stratum $S\in \cS$ is a  \emph{fountain stratum} in $(X,\cS,\cT)$ if  $\dim S < \dim S^\iota$.
  \end{definition}
By dimensional reasons, any stratum $T\in\cT$ admits a source stratum $S\in\cS$; i.e., $T=S^\iota$.
Moreover,  the union of source strata of $T$ 
is always an open dense subset of $T$.
Exceptional strata verify the inequality on dimensions but their particularity in the conservation of intersection homotopy 
groups of coarsenings makes us treat them separately. 

\medskip 
 Let us specify the different local situations in a simple CS-coarsening, as it appears in \cite[Proposition 4.4]{MR4170473}. 
 
 \begin{definition}\label{def:simple}
 A CS-coarsening  $(X,\cS,\cT)$ is \emph{simple} if, in the poset $\cV$ formed of   exceptional strata  and  fountain strata,
the relation $S_{1}\preceq S_{2}$ implies $S_{1}=S_{2}$.
 \end{definition}
 
 \noindent
 Suppose $(X,\cS,\cT)$ is a simple CS-coarsening.
 
 \medskip
$\bullet$ Let $S\in\cS$ be an \emph{exceptional stratum} and $x\in S\subset S^\iota$. We already observed that
 $b=\dim S^\iota-\dim S\geq 1$. There exist an $\cS$-conical chart $(\phi,W)$ of $x\in S$ whose link is
 $(\S^{b-1},\cI_{\S^{b-1}})$, a $\cT$-conical chart $(\psi, W)$ of $x\in S^\iota$, and a commutative diagram where $f$ is a homeomorphism,
\begin{equation}\label{equa:exceptional}
\xymatrix{
(\R^a \times \rc \S^{b-1}, \I_{\R^a} \times \rc \mathcal I_{\S^{b-1}})  \ar[d]_{f}\ar[r]^-\phi  & (W,\mathcal S) \ar[d]^\iota\ \\
(\R^{a+b}, \I_{\R^{a+b}} ) \ar[r]^-\psi & (W,\mathcal T).
}
\end{equation}
In reference to \defref{def:startifiedmap}, let us notice that the two horizontal maps, $\phi$ and $\psi$, are stratified homeomorphisms,
in contrast with the vertical maps that are only homeomorphisms and stratified.

\medskip
$\bullet$  Let $S\in\cS$ be a \emph{source stratum} and $x\in S\subset S^\iota$. 
There exist an $\cS$-conical chart $(\phi,W)$ of $x\in S$ whose link is $(L,\mathcal L)$, 
a $\cT$-conical chart $(\psi, W)$ of $x\in S^\iota$ whose  link is $(L,\L')$ for    some coarsening  $\L'$ of $\L$,
and a commutative diagram,
\begin{equation}\label{equa:source}
\xymatrix{
(\R^a \times \rc L, \I_{\R^a} \times  \rc \mathcal L)  \ar[d]_{\iota} \ar[r]^-\psi  & (W,\mathcal S) \ar[d]^\iota\ \\
(\R^a \times \rc L, \I_{\R^a} \times \rc \mathcal L') \ar[r]^-\psi & (W,\mathcal T).
}
\end{equation}
In the diagrams \eqref{equa:exceptional} and \eqref{equa:source}, the map $\iota$ is the identity as a map of spaces.

\medskip
For the   description of the conical charts of a point in a fountain stratum, we need to recall the join of 
two filtered spaces in a particular case.

\begin{example}\label{exam:joint} 
Let $(X,\cS)$ be a filtered space and $\S^m$ be a homological sphere, $m\geq 0$. 
We consider the join defined as a quotient, $\S^{m} * X = \tc\S^{m}\times X/\sim$, 
 with an equivalence relation generated by 
 $
 ([z,1],x) \sim ([z,1],x').
 $
 The elements of  the join $\S^{m} * X$ are denoted by $[[z,t],x]$. 
 We identify $\S^{m}$ with the subspace formed of  classes
$[[z,1],x]   $ and $X$ with the subspace of  classes $[[z,0],x]$. 
The product $\S^m \times \rc X$ is 
also naturally included in the join through the map $(z,[x,t]) \mapsto [[z,1-t],x]]$.
We endow  $\S^m*X$ with the \emph{join filtration} defined by
$$
\S^{m} \subset \S^{m} *X_0 \subset\cdots \subset \S^{m} * X_{n-1} \subsetneq \S^{m} * X_{n}.
$$
The corresponding set of strata, called the \emph{join stratification} and denoted by $\S^m*\cS$, has for elements
the connected components of $\S^m$ and the products $\tc\S^m\times S$ with $S\in \cS$. 
The previous inclusions of $\S^m$, $X$ and $\S^m\times \rc X$ in $\S^m*X$ induce stratified maps.

\medskip 
 The cone on a join can be decomposed as a product via a homeomorphism
 $h\colon \rc(\S^m*X)\to \rc \S^m \times \rc X$, see \cite[5.7.4]{MR2273730}.
 To understand the behavior of $h$ with respect to stratifications, we  introduce the refinement 
 of a stratification by adding a point. More specifically, let $(X,\cS)$ be a filtered space and $x$ be a point of $X$
 such that $\{x \} \not\in \SS$. 
We define a new filtration by adding $\{x\}$ to the bottom of the filtration. We get
$$ 
X_0 \cup \{ x \} \subset X_1\cup \{ x \}   \subset \cdots \subset  X_n \cup \{ x \} =X,
$$
whose associated stratification, denoted by $\cS_{x}$, has for elements 
$\{x\}$ and the connected components of $S\menos \{x\} $ with $S\in\cS$.

\medskip
Denote by $\tu$  the apex of $\rc \S^m$ and by $\tv$ that of $\rc X$.
  The homeomorphism $h$ induces (see \cite[Proposition 3.3]{MR4170473}) a stratified homeomorphism
 \begin{equation}\label{equa:hoeoconejoin}
 h\colon ((\rc (\S^{m} * X), \rc({\S^m}*\cS)) \to (\rc \S^m \times \rc X, (\I_{\rc\S^m} \times \rc \cS)_{(\tu,\tv)}).
 \end{equation}
\end{example}

\medskip
$\bullet$ Let $S\in\cS$ be a \emph{fountain stratum} and $x\in S\subset S^\iota$. 
Denote $b = \dim S^\iota - \dim S \geq 1$. Then, there exist
a $\T$-conical chart   $(\psi,W)$ of $x  \in S^\iota$ whose link is $(L,\cL)$,
an  $\cS$-conical chart  $(\phi, W)$ of  $x   \in S$ whose link is $  (\S^{b-1}*L, {\S^{b-1}}*\cL)$,
and a commutative diagram where $J$ is a homeomorphism,
\begin{equation}\label{equa:virtual}
\xymatrix{
(\R^a \times \rc (\S^{b-1} * L), \I_{\R^a} \times  \rc({\S^{b-1}} * \cL))  \ar[d]_{J} \ar[r]^-\phi  & (W, \cS) \ar[d]^\iota\ \\
(\R^{a+b} \times \rc L, \I_{\R^{a+b}} \times \rc  \cL) \ar[r]^-\psi & (W, \cT).\\
}
\end{equation}
The homeomorphism $J$ is equal to (\cite[Proposition 2.7]{MR4170473})
 $J = (f\times \Id)\circ (\Id\times h)$
where $h \colon \rc( \S^{b-1} \ast L )\to \rc \S^{b-1} \times \rc L$ 
 is the stratified homeomorphism defined in \eqref{equa:hoeoconejoin}  
and the homeomorphism $f \colon (\R^a \times \rc \S^{b-1},(0,\tv))   \to (\R^{a+b},0)$ already appears in \eqref{equa:exceptional}.
Diagram \eqref{equa:virtual} can therefore  be presented in the following alternative form,
where $\tu$  is the apex of   $\rc \S^{b-1}$ and  $\tv$ the apex of $\rc L$,

 \begin{equation}\label{equa:virtual2}
\xymatrix{
(\R^a \times  (\rc \S^{b-1} \times \rc L), \cI_{\R^a} \times (\cI_{\rc\S^{b-1}} \times \rc \cL)_{(\tu,\tv)}) \ar[d]_-{\iota} \ar[r]^-{\phi'}  
&
 (W,\cS) \ar[d]^-{\iota}\ \\
((\R^a \times \rc \S^{b-1} )\times  \rc L, \cI_{\R^a\times \rc S^{b-1}}   \times \rc \mathcal L) \ar[r]^-{\psi'} & (W,\cT).
}
\end{equation}
Both maps $\psi'=\psi\circ (f\times \Id)$ and $\phi'=\phi\circ (\Id\times h^{-1}) $ are stratified homeomorphisms.
Concerning the stratifications, the space 
$\R^a \times \rc \S^{b-1} \times \{ \tv \} $ is a stratum on the bottom line.
On the top line, this space is the union of the  strata $\R^a \times \{(\tu,\tv)\}$ and 
$\R^a \times (\rc \S^{b-1}\menos \{ \tu\})_{cc} \times \{ \tv \}$, 
where $(\rc \S^{b-1}\menos \{ \tu\})_{cc} $ denotes any connected component of $\rc \S^{b-1}\menos \{ \tu\}$.
In short, we have 
\begin{eqnarray}
\R^a \times \{(\tu,\tv)\}
&\preceq&
\R^a \times (\rc \S^{b-1}\menos \{ \tu\})_{cc} \times \{ \tv \},\label{equa:virtualiota}\\
\R^a \times \rc \S^{b-1} \times \{ \tv \} &=&(\R^a \times \{(\tu,\tv)\})^\iota
= (\R^a \times (\rc \S^{b-1}\menos \{ \tu\})_{cc} \times \{ \tv \}.)^\iota.\label{equa:virtualiota2}
\end{eqnarray}

\medskip
Let $(X,\cS,\cT)$ be a CS-coarsening.
The main interest of the concept of  simple coarsening lies in the existence of a decomposition of the identity 
map, $\iota\colon (X,\cS)\to (X,\cT)$, into a sequence of simple coarsenings, established in
\cite[Proposition 2.3]{MR4170473}. Let us use it to study the normality of a coarsening.

\begin{proposition}\label{prop:normal}
Let $(X,\cS,\cT)$ be a CS-coarsening. If $(X,\cS)$ is a normal CS set, then so is $(X,\cT)$.
\end{proposition}

\begin{proof}
From \cite[Proposition 2.3]{MR4742271}, we may suppose that the CS-coarsening is simple.
Recall that if one link of a point in a CS set is connected, then all links of this point are connected 
(\cite[Remark 2.6.2]{LibroGreg}). 
Thus, we have to prove that any singular point $x$ of $(X,\cT)$ has a connected link in $(X,\cT)$.
Let $T\in\cT$ and $S\in\cS$ be the strata containing $x$. We distinguish two cases.
\begin{itemize}
\item If $S$ is a source stratum of $\cT$, we are in the situation \eqref{equa:source}.
Therefore a link of $x$ in $(X,\cT)$ is a link of $(X,\cS)$ which is connected by hypothesis.
\item Suppose that $S$ is not a source stratum of $T$. Since the union of source strata is a dense subset of $X$,
we may assume that there exists a point $x'$ sufficiently close to $x$ that belongs to a source stratum.
By the previous case, every link of $x'$ in $(X,\cT)$ is connected. 
Since both points are close, we can suppose that they belong to the same conical neighborhood of $x$,
$$\varphi\colon (U\times \rc L, \I_{U}\times \rc \cL)\to (V,\cT).$$
Thus $x$ and $x'$ have a common link, $L$, in $(X,\cT)$, which is connected.
\end{itemize}
\end{proof}

\begin{remark}\label{rem:normalpas1}
If $(X,\cS,\cT)$ is a coarsening with $(X,\cS)$ normal, then there is no exceptional strata of codimension 1.
Indeed, if $S\in\cS$ is exceptional, any link of a point $x\in S$ is a homological sphere. By connectivity,
this sphere must be of dimension $\geq 1$. Thus the stratum $S$ must be of codimension $\geq 2$.
\end{remark}
\section{Perversities}\label{sec:perversity}
\begin{quote}
 The notion of perversity on a filtered space is the basis of intersection homology and  homotopy.
  Here we recall definition and properties of general perversities of MacPherson (\cite{RobertSF}) 
  and introduce the  $K$-perversities for which a CS-coarsening invariance holds.
 \end{quote}

\begin{definition}\label{def:perversity}
A \emph{perversity} on a filtered space $(X,\cS)$ is a map $\ov p \colon \cS \to  \ov \Z = \Z \sqcup \{-\infty,\infty\}$ 
verifying $\ov p(S)=0$ for any regular stratum. 
The \emph{top perversity} is  defined by $\ov t (S) =\codim\, S - 2$ for each singular stratum $S$. The \emph{dual perversity} $D\ov p$ of  a perversity $\ov p$ is   defined by $D\ov p = \ov t - \ov p$.
A triple $(X,\cS,\ov{p})$ of a filtered space with a perversity on it is called a \emph{perverse space.}
\end{definition}

To any  map $f\colon \N\to\Z$ such that $f(0)=0$, we associate a perversity defined by  $\ov{p}(S)=f(\codim\, S)$.  
Such perversities depend only on the codimension of the stratum and 
are called \emph{codimensional perversities}.
If $f(\ell)=k$ for any $\ell\neq 0$, the associated perversity is denoted by $\ov k$.
The original perversities of \cite{GM1} and those of \cite{MR800845} are codimensional. 
We will come back to them in \secref{sec:refinementexamples}.

\smallskip
Perversities can be transferred through a stratified map $\varphi\colon (X,\cS)\to (Y,\cT)$. 

\begin{definition}\label{def:pushback}
Let $\ov{q}$ be a perversity on $(Y,\cT)$. The  \emph{pullback perversity} $\varphi^\star\ov q$ on $(X,\cS)$ is 
defined by  $\varphi^\star\ov q (S) = \ov q (S^\varphi)$ for each singular stratum $S \in \cS$.
  Let $\ov{p}$ be a perversity on $(X,\cS)$. 
The \emph{pushforward perversity} $\varphi_\star\ov p$ on $(Y,\cT)$ 
 is defined by $\varphi_\star\ov p (T) = \inf\{ \ov p ( Q ) \mid  Q^\varphi = T\}$ for each singular $T \in \cT$, with 
  $\inf \emptyset =\infty$. 
\end{definition}

Notice that  $\varphi_\star \varphi^\star \ov q = \ov q$ and $\varphi^\star \varphi_\star \ov p  \leq  \ov p$.
(When $\varphi$ is a stratified homeomorphism, we have 
$\varphi_\star \ov p (S^\varphi) =  \ov p (S)$ for each $S \in \cS$ and $\varphi^\star \varphi_\star \ov p  =  \ov p$.)
In the following, the map $\varphi$ is induced by the identity. We denote $\iota_{\cS,\cT}\colon (X,\cS)\to (X,\cT)$
(or $\iota$ if there is no ambiguity) the induced stratified map when $(X,\cS,\cT)$ is a CS-coarsening.

\bigskip
The perversities of \cite{GM1,MR800845} verify a key property for the topological invariance of intersection homology, called \emph{the growing property}. 
A codimensional perversity verifies the growing property if  
the following inequalities are verified for each integer $i$,
\begin{equation}\label{equa:growing}
\ov p(i)\leq \ov p(i+1)\leq \ov p(i)+1.
\end{equation}
The next definition is an adaptation of  these inequalities to  perversities defined on the set of strata.
They are the perversities for which a CS-coarsening invariance holds.

\begin{definition}\label{def:Kper}
 A   \emph{$K$-perversity} $\ov p$ on  a CS-coarsening $(X,\cS,\cT)$  is a
 perversity on $(X,\cS)$ verifying the  conditions (K1) and (K2).
\begin{enumerate}[(K1)]
\item For any strata $S,Q \in \cS$ such that $S\preceq Q$ and $S^\iota= Q^\iota$, one has
\begin{equation}\label{eq:magica}
 \ov p(Q) \leq \ov p(S) \leq \ov p(Q) + \ov t(S) - \ov t(Q).
\end{equation}
\item For any strata $S,Q \in \cS$ with $\dim S =\dim Q$ and $S^\iota= Q^\iota$, one has
\begin{equation}\label{eq:magicaTris}
 \ov p(Q) = \ov p(S).
\end{equation}
\end{enumerate}
\end{definition}

\begin{remark}\label{rem:propertyKperversity}
We list some basic observations on $K$-perversities.

\smallskip
i) The inequalities \eqref{eq:magica}  are equivalent to 
$$\ov p(Q)\leq \ov p(S)\;\text{ and } D\ov p(Q)\leq D\ov p(S).$$
With the hypothesis $\dim S=\dim Q$, the equality \eqref{eq:magicaTris} is equivalent to $D\ov p(Q)=D\ov p(S)$.
So, the perversities $\ov p$ and   $D\ov p$ are  simultaneously  $K$-perversities.

\smallskip
ii) Two regular strata, $S$ and $Q$,  always verify the condition  \eqref{eq:magica}.
 If the stratum $S$ is  exceptional, as a singular stratum is always included in the closure of a regular one, say $Q$, 
 we have
 $S\preceq Q$ and $S^\iota=Q^\iota$. So the  condition \eqref{eq:magica} gives
\begin{equation}\label{eq:regS}
0 \leq \ov p(S) \leq \ov t(S),
\end{equation}
which implies the following properties.
\begin{itemize}
\item A $K$-perversity takes positive values on exceptional strata. 
\item The existence of a $K$-perversity implies the \emph{non-existence of exceptional strata of codimension~1,}
   since $0\leq \ov t(S)=-1$ is impossible.
\end{itemize}

\smallskip
iii) Finally, if $(X,\cS,\cX^\star)$ is  the intrinsic CS-coarsening   (\cite{MR800845}),  
 \defref{def:Kper} is equivalent to the definition of $K$-perversity
given in \cite[Definition 6.8]{CST3}.
\end{remark}

We specify the behaviors of $K$-perversities with  pullback and pushforward processes. 

\begin{lemma}\label{LemTec}
Let $\ov{p}$ be a $K$-perversity on the CS-coarsening $(X,\cS,\cT)$. Then, the following properties are verified:
\begin{enumerate}[i)]
\item $\iota_{\star}\ov p(T)=\ov p(S)$ if $S\in\cS$ is a source stratum of $T\in\cT$,
\item   $\iota^\star  D\iota_\star\ov p \leq D\ov p$,
\item  $\iota_\star D \ov p = D \iota_\star \ov p$,
\item if $\ov p\leq \ov t$, then $\iota_{{\star}}\ov p\leq \ov t$.
\end{enumerate}
\end{lemma}

\begin{proof}
Properties i) and ii) are exactly \cite[Lemma 5.3 and 5.4]{MR4170473}.
It remains to prove iii) and iv). 
Let $T \in \TT$ be a singular stratum. We know that there is a source stratum  $S \in  \cS$ of $T$. 
Since $D\ov p$ is a $K$-perversity, we have from i),
$\iota_{\star}D \ov p(T)=D\ov p(S)$.
By replacing dual perversities with their definition and using i) a second time, we get iii) since
$$\iota_{\star}D\ov p(T)=\ov t(S)-\ov p(S)=\ov t(T)-\iota_{\star}\ov p(T)=D\iota_{\star}\ov p(T).$$
Similarly, the last property is a consequence of
$$\iota_{{\star}}\ov p(T)=\ov p(S)\leq \ov t(S)=\codim\,S-2=\codim\,T-2=\ov t(T).$$
\end{proof}

\begin{lemma}[{\cite[Lemma 5.6]{MR4170473}}\label{DI}] 
Let $(X,\cS,\cR)$ and $(X,\cR,\cT)$ be two CS-coarsenings. 
If  $\ov p$ is a $K$-perversity on the CS-coarsening $(X,\cS,\cT)$, 
then the following properties are verified.
\begin{enumerate}[i)]
\item The perversity $\ov p$ is a $K$-perversity on the CS-coarsening $(X,\cS,\cR)$.
\item The perversity $\iota_\star\ov p$ is a $K$-perversity on the CS-coarsening $(X,\cR,\cT)$. 
\end{enumerate}
\end{lemma}

\begin{lemma}\label{Dback} 
Let $(X,\cS,\cT)$  be a CS-coarsening without exceptional strata of codimension 1. 
If  $\ov q$ is a perversity on   $(X,\cT)$, then the pullback $\iota^\star\ov q$
is a $K$-perversity on the CS-coarsening $(X,\cS,\cT)$.
Moreover, the inequality $\ov q\leq \ov t$ implies $\iota^\star\ov q\leq \ov t$.
\end{lemma}

\begin{proof}
Let $S,\,Q\in \cS$ such that $S^\iota=Q^\iota$. We have to prove (K1) and (K2).
Suppose $S\preceq Q$. With the definition of $\iota^\star\ov q$, the inequalities \eqref{eq:magica} of (K1) become
$$
\ov q(Q^\iota)\leq \ov q (S^\iota)\leq \ov q(Q^\iota)+\ov t(S)-\ov t(Q),$$
which is equivalent to
$\ov t(S)-\ov t(Q)\geq 0$. This last inequality is true if $S$ and $Q$ are both regular or both singular. 
There remains the exceptional case for $S$, where the previous inequality is reduced to 
$\ov t(S)=\codim\, S-2\geq 0$. This is true since there is no exceptional strata of codimension 1.

\smallskip
For (K2), we suppose $\dim S=\dim Q$. By definition of $\iota^\star\ov q$, we have, as expected,
$$
\iota^\star\ov q(Q)=\ov q(Q^\iota)=\ov q(S^\iota)=\iota^\star\ov q(S).
$$

\smallskip
Suppose $\ov q\leq \ov t$ and let $S\in\cS$ be  a singular stratum. If $S^\iota$ is a regular stratum, then
$\iota^\star\ov q(S)=\ov q(S^\iota)=0\leq \ov t(S)$, since there is no exceptional strata of codimension 1. 
If $S^\iota$ is singular, we have:
$$
\iota^\star \ov q(S) =  \ov q(S^\iota) \leq  \ov t(S^\iota) = \codim (S^\iota) -2\leq \codim (S) -2 =
 \ov t(S).
$$
\end{proof}

\begin{remark}\label{rem:perversityonlink}
Let us specify how some particular induced perversities are defined. 

\smallskip
(i) The first point concerns the induced perversity on a link $L$, from a given perversity $\ov p$ on a CS set $X$.
If  $ \varphi\colon \R^k\times \rc L\to V$ is a conical chart of $x\in S\in \cS$, the
open subset $V\menos S$  of $X$  is given with the induced perversity still denoted by $\ov{p}$. 
If $T$ is a stratum of $L$, then $\R^k\times ]0,1[\times T$ is a stratum of $V\menos S$ and we set
$\ov{p}(T)=\ov{p}(\R^k\times [0,1[\times T)$. We also set $\ov{p}(\{\tv\})=\ov{p}(S)$ for the apex $\tv$ of $\rc L$.

\smallskip
(ii) The second point concerns the left vertical arrow of \eqref{equa:virtual2}, with the identity map,
$$ \iota\colon
A:=(\R^a \times  (\rc \S^{b-1} \times \rc L), \cI_{\R^a} \times (\cI_{\rc\S^{b-1}} \times \rc \cL)_{(\tu,\tv)})
\to
B:= ((\R^a \times \rc \S^{b-1} )\times  \rc L, \cI_{\R^a\times \rc S^{b-1}}   \times \rc \mathcal L).
$$
We start with a $K$-perversity $\ov p$ on $A$ and denote $\ov q=\iota_{{\star}}\ov p$ its pushforward perversity on $B$.
Let us notice that $D\ov q$ induces  a perversity on $(\rc L,\rc \cL)$ as in the  point (i).
From \lemref{LemTec}, we know that $D\ov q=\iota_{{\star}}D\ov p$.
The strata
$S=\R^a\times\{(\tu,\tv)\}$
and
$Q=\R^a\times \left(\rc \S^{b-1}\menos\{\tu\}\right)\times\{\tv\}$
verify
$$S\preceq Q \;\text{and}\;S^\iota=Q^\iota=\R^a\times\rc \S^{b-1}\times\{\tv\}.
$$
As $D\ov p$ is a K-perversity (\defref{def:Kper} and \remref{rem:propertyKperversity}), we have
\begin{equation}\label{eqmagicater}
D\ov p(Q)\leq D \ov p(S)\leq D\ov p(Q)+\ov t(S)-\ov t(Q)=D\ov p(Q)+b.
\end{equation}
In this particular case, the first item of this remark gives
$D\ov p(S)=D\ov p(\{\tu,\tv\})$ and
$D\ov p(Q)=D\ov p(\{\tv\})$.
With these values, the inequalities \eqref{eqmagicater} become
\begin{equation}\label{eq:AB}
D\ov p(\{\tv\})\leq D\ov p(\{(\tu,\tv)\})\leq D\ov p(\{\tv\})+b.
\end{equation}
Let us also notice that the two stratifications coincide on the open subset
$$\R^a\times (\tc \S^{b-1}\times \rc L)\menos (\R^a\times \{(\tu,\tv)\}).$$
Thus the two perversities
 $D\ov p$ and $D\ov q$ also coincide on this subset.
\end{remark}

\section{Intersection homotopy groups}\label{sec:inthomo}
 \begin{quote}
Intersection homotopy groups of a filtered space $(X,\cS)$, endowed with a perversity,  are defined 
as the homotopy groups of a simplicial set introduced by P. Gajer (\cite{Gajer1}).
In this section, we recall their definitions and properties used in the following.
\end{quote}
 
 Let $(X,\cS,\ov p)$ be a perverse space. 
 A simplex $\sigma \colon \Delta \to X$ is  
\emph{$\ov p$-allowable} if we have, for each \emph{singular} stratum $S\in\cS$,
\begin{equation}\label{allow}
\dim \sigma^{-1}(S) \leq \dim \Delta - D\ov p(S) -2.
\end{equation}
Here the dimension chosen for $\dim \sigma^{-1}(S)$ is the polyhedral dimension (\cite{Gajer1}).
We do not need the definition; 
it is enough to know that it gives the same properties as the classical dimension by the skeleton, see \cite{CST8}.
The simplex $\sigma$ is \emph{$\ov p$-full} if each of its faces is $\ov p$-allowable.
 The set of $\ov{p}$-full simplices is a simplicial set verifying the Kan condition, see
 \cite[Page 946]{Gajer1}, \cite[Proposition 2.3]{CST9}.
 We  denote it by $\crG_{\ov{p}}(X,\cS)$ (or $\crG_{\ov{p}}X$ if there is no ambiguity)
 and call it the \emph{Gajer $\ov{p}$-space} associated to $(X,\cS)$.

\begin{proposition}[{\cite[Proposition 3.4]{CST9}}]\label{prop:homotopyxR}
 Let $(X,\cS, \ov{p})$ be a  perverse space 
  and $Y$ be a  topological space. 
 The product $Y\times X$ is equipped with the product filtration $\cI_{Y}\times \cS$ (\exemsref{exam:stratifications})
 and with the perversity, still denoted by $\ov{p}$,
 defined by $\ov{p}(Y\times S)=\ov{p}(S)$.
 Then, the canonical projections, $p_{Y}\colon Y\times X\to Y$ and $p_{X}\colon Y\times X\to X$,
 induce an isomorphism 
 $$\crG_{\ov{p}}(Y\times X)\cong \Sing \,Y\times \crG_{\ov{p}}X,$$
 where $\Sing \,Y$ is the simplicial set made up of the singular simplices of $Y$.
\end{proposition}

Let $f\colon (X,\cS,\ov{p})\to (Y,\cT,\ov q)$ be a stratified map between perverse filtered spaces such that
$f^\star D\ov q\leq D\ov p$. Then $f$ induces a simplicial map
$\crG (f)\colon\crG_{\ov p}(X,\cS)\to\crG_{\ov q}(Y,\cT)$, see \cite[Proposition 3.5]{CST9}.
In particular, if $(X,\cS,\cT)$ is a CS-coarsening and $\ov p$ a perversity on $(X,\cS)$, we have
$\iota^\star D\iota_{\star}\ov{p}\leq D\ov p$, by \lemref{LemTec}.ii). So the identity map
induces a simplicial map
\begin{equation}\label{equa:gajermap}
\crG (\iota)\colon \crG_{\ov p}(X,\cS)\to \crG_{\iota_{\star}\ov p}(X,\cT).
\end{equation}

A \emph{pointed perverse space} is a quadruplet $(X,\cS,\ov{p},x_{0})$ 
where $(X,\cS,\ov{p})$ is a perverse space and 
$x_{0}\in X\menos \Sigma$ is a regular point, called  \emph{basepoint}.
We also denote by $x_{0}$ the simplicial subset of $\Sing\, X$ generated by $x_{0}\colon \Delta^0\to X$.

\begin{definition} [{\cite{Gajer1}, \cite[Definition 4.1]{CST9}}]\label{def:intersectionhomotopygp}
Let $(X,\cS,\ov{p},x_{0})$ be a pointed perverse space.
The \emph{$\ov{p}$-intersection homotopy  groups}
are the homotopy groups of the simplicial set $\crG_{\ov{p}}(X,\cS)$,
$$\pi^{\ov{p}}_{\ell}(X,\cS,x_{0}):=\pi_{\ell}(\crG_{\ov{p}}(X,\cS),x_{0}).$$
\end{definition}

If there is no ambiguity, we do not mention the basepoint and write $\pi_{\ell}^{\ov p}(X,\cS)$.
Intersection homotopy groups of cones satisfy a property similar to that of intersection homology 
(\cite[Proposition 4.5]{CST9}): suppose $X$ compact and let $\rc X$ be the open cone of apex $\tv$, 
with the cone filtration and pointed by $y_{0}=[x_{0},1/2]$.
Let $\ov{p}$ be a   perversity on $\rc X$,  
we denote also by $\ov{p}$ the perversity induced on $X$.
Then, the $\ov{p}$-intersection homotopy groups of $\rc X$ are given by 
\begin{equation}\label{equa:cone}
\pi_{\ell}^{\ov{p}}(\rc X,\rc\cS,y_{0})=\left\{
\begin{array}{lcl}
\pi_{\ell}^{\ov{p}}(X,\cS,x_{0})&\text{ if }& \ell\leq D\ov{p}(\tv),\\ 
0&\text{ if }& \ell> D\ov{p}(\tv). 
\end{array}
\right.
\end{equation}

The homology of $\crG_{\ov p}(\rc X,\rc S)$ has a different behavior (\cite[Proposition 3.10]{CST9}).
Recall that a local coefficient system (of abelian groups) on a simplicial set $K$ is a contravariant functor
$\gE\colon \Pi K\to \mathrm{Ab}$, from its fundamental groupoid,  with values in the category of abelian groups.
The inclusion $L\hookrightarrow \rc L$ given by  $(x\mapsto [x,1/2])$, induces 
\begin{equation}\label{equa:homologyGajercone}
\left\{\begin{array}{lrccl}
\hbox{isomorphisms} &
H_{\ell}(\crG_{\ov{p}}L;\gE) & \xrightarrow{\cong} &
H_{\ell}(\crG_{\ov{p}} \rc L;\gE ), &  \text{if}\quad \ell\leq D\ov{p}(\tv) \hbox{ and }\\
\text{an epimorphism} &
H_{\ell}(\crG_{\ov{p}}L;\gE) & \twoheadrightarrow &
H_{\ell}(\crG_{\ov{p}} \rc L;\gE), &  \text{if}\quad \ell = D\ov{p}(\tv) +1 .
\end{array}\right.
\end{equation}
In other words, 
the classical \emph{relative} homology verifies
\begin{equation}\label{equa:conegajer}
H_{j}(\crG_{\ov{p}}\rc L, \crG_{\ov{p}} L ;\gE ) =
0\quad\text{for any}\quad j\leq D\ov{p}(\tv)+1.
\end{equation}
%

 \section{Connectivity and CS-coarsenings}\label{sec:connected}
 \begin{quote}
 In this section, we study the set $\pi_{0}^{\ov{p}}(X)$ of  connected components of the Gajer space $\crG_{\ov p}X$
 associated to a perverse space $(X,\cS,\ov{p})$.
The goal is to demonstrate, in \propref{lemapi1}, the existence of an isomorphism 
between the sets of connected components
of the two stratifications in a CS-coarsening. This is the first condition to fill in 
for an application of  \lemref{lem:quillenfini}.
 \end{quote}
 
 \begin{lemma}\label{lem:pi1surj}
 Let $(X,\cS,\ov{p},x_{0})$ be a pointed perverse space with  $\ov{p}\leq \ov t$. Then, the canonical inclusion induces an isomorphism of sets  and a surjective homomorphism,
\begin{equation}\label{equa:pi0pi1surj}
\pi_{0}(X\menos \Sigma)\xrightarrow{\cong} \pi_{0}^{\ov p}(X,\cS) \quad \text{and} \quad
\pi_{1}(X\menos \Sigma,x_{0})\twoheadrightarrow \pi_{1}^{\ov p}(X,\cS,x_{0}).
\end{equation}
\end{lemma}
\begin{proof}

This comes directly from the definition of  $\ov{p}$-full simplices: 
in dimension 0 or 1 they do not meet the singular part of $X$.
\end{proof}
 A perverse space is said \emph{$\ov{p}$-connected} if $\pi_{0}^{\ov p}(X,\cS)$
 consists of a single element.
When $(X,\cS,\ov p)$ is $\ov p$-connected, we do not need to specify the basepoint  
of intersection homotopy groups
 and   we  write $\pi^{\ov p}_\ell (X,\cS)$ instead of $\pi^{\ov{p}}_{\ell}(X,\cS,x_{0})$.

 \begin{lemma}\label{lem:cheminreg}
Let $(X,\cS)$ be a  CS set.  For any point $x \in X$, there exists a   path 
$\beta \colon [0,1] \to X$ with $ \beta(0) =x$ and $\beta(]0,1]) \subset X\menos \Sigma$. 
This path can be chosen as small as wanted.
Moreover, if $x$ is $\ov{p}$-allowable for some perversity $\ov p$, then $\beta$ is $\ov{p}$-allowable too.
\end{lemma}

\begin{proof}
If the point $x$ is regular, we choose the constant path $x$. 
If not, let   $\varphi \colon \R^k \times \rc L \to U \subset X$ be a conical chart  with $\varphi(0,\tv)=x$. 
(Here, $\tv$ is the apex of the cone $\rc L$.) 
We choose a point $z \in L\menos \Sigma_L$, which exists since the regular part of $L$ is 
not empty.
For any $\varepsilon >0$, the map $\beta(t) = \varphi (0, [z,\varepsilon \, t])$
checks the required conditions.
Let $S$ be the stratum of $\cS$ containing $x$, the last assertion is a consequence of
$$
\dim \beta^{-1}(S) = \dim \{x\} = 0 \leq \dim \{x\} -D\ov p(S) -2 = \dim \beta^{-1}(S) -D\ov p(S) -2.
$$
\end{proof}

The next result is demonstrated in  \cite[Lemma 2.6.3]{LibroGreg}  for pseudomanifolds.
 For these, a link is also a pseudomanifold  and the proof can be done by induction on the dimension. 
 Since CS sets do not have such a recursive property, we provide an adaptation of it.

\begin{lemma}\label{lem:nor}
Let $(X,\cS)$ be a normal CS set. 
Then the space $X$  is connected if and only if the regular part $X\menos \Sigma$ is.
\end{lemma}

\begin{proof}
Since the regular part is dense in $X$, we have only  one implication to prove.
So let $(X,\cS)$ be a normal, connected CS set. 
We proceed by induction on the depth.
When $\depth(X,\cS)=0$, we have  $\Sigma=\emptyset$, and the result is clear.
Suppose the result is true for any CS set of depth $<n$ and  that 
$(X,\cS)$ is a normal, connected CS set, of depth $n$,  with more than one regular stratum.
We denote by $S$ one regular stratum and by $T$ the union of the other regular strata.
We know $X=\ov S\cup \ov T$ and the existence of\  $x\in \ov S\cap \ov T$ since $X$ is connected.
Let $\varphi\colon\R^i\times \rc L\to U$ be a conical chart of $x$. The intersection of $S$ and $T$
with $U$ are of the form
$\R^i\times A\times ]0,1[$ and $\R^i\times B\times ]0,1[$,
where $A$ and $B$ are union of regular strata of $L$.
Since $L$ is a filtered space, $A$ and $B$ are open in $L$ and the CS set 
$\R^i\times L\times ]0,1[$
is a CS set of depth $<n$ with more than one regular stratum.
We get the result by contradiction.
\end{proof}

\begin{proposition}\label{lemapi1}
Let $\ov{p}$ be a $K$-perversity on a CS-coarsening  $(X,\cS,\cT)$, such that $\ov p \leq \ov t$.
Then, the identity map induces a bijection
$$
\iota_0 \colon \pi_0^{\ov p}(X,\cS)\xrightarrow{\cong} \pi_0^{\iota_\star \ov p}(X,\cT).
$$
\end{proposition}
\begin{proof}
From \eqref{equa:gajermap}, we know the existence of a simplicial map
$$\crG(\iota) \colon \crG_{\ov{p}}(X,\cS)\to \crG_{\iota_\star \ov p}(X,\cT).$$
We denote the induced maps by 
$\iota_\ell \colon \pi_\ell^{\ov p}(X,\cS) \to  \pi_\ell^{\iota_\star \ov p}(X,\cT)
$.
By hypothesis, we have $D\ov{p}\geq \ov{0}$.
From $\iota_{\star}D\ov p=D\iota_{\star}\ov p$ (\lemref{LemTec}), we deduce $D\iota_{\star}\ov{p}\geq \ov{0}$.
So, the allowability condition \eqref{allow} of an  $\ell$-simplex, with $\ell=0,\,1$, gives 
$$\dim\sigma^{-1}(S)\leq \dim\sigma-D\ov p(S)-2<0.$$
Thus the $\ov p$-full $\ell$-simplices of $(X,\cS)$ and $(X,\cT)$
do not meet the respective singular parts for $\ell=0,\,1$.
So, with \lemref{lem:pi1surj}, the desired property becomes 
\begin{equation}\label{nada}
\iota_0 \colon \pi_0(X\menos \Sigma_\SS) \longrightarrow \pi_0 (X\menos \Sigma_\TT) \hbox{ is a bijection.}
\end{equation}
As $Y = X\menos \Sigma_\cT$ is an open subset, there is an induced CS set structure  $(Y,\cS)$. 
The  condition \eqref{nada}  becomes 
$$
\iota_0 \colon  \pi_0(Y \menos \Sigma_\cS) \longrightarrow  \pi_0(Y)  \hbox{ is a bijection.}
$$
 Without loss of generality, we may assume that $Y$ is connected. 
 All  strata of $(Y,\cS)$ are exceptional in $(X,\cS,\cT)$ and none are of  codimension 1 (cf. \eqref{eq:regS}). 
  From \defref{def:csset} and  diagram \eqref{equa:exceptional}, we deduce that 
 any link of any exceptional stratum is a homology sphere of dimension $\geq 1$.
 Hence, the CS set $(Y,\cS)$ is normal.
 Finally, \lemref{lem:nor}  implies that $Y\menos \Sigma_\cS$ is connected and 
 therefore  property  \eqref{nada} is true.
 \end{proof}
 
 \begin{proposition}\label{prop:connectionLink}
 Let $(X,\cS,\ov p)$ be a connected, normal perverse CS set.
  Then,  any link $(L,\cL)$ is $\ov p$-connected; i.e., $\pi_{0}^{\ov p}(L,\cL)$ consists of a single element.
 \end{proposition}
 
 \begin{proof}
 Let $\varphi\colon \R^m \times \rc L \to U$ be a conical chart. As $\R^m\times L\times ]0,1[$ is an open subset
 of $X$, it is a CS set and any of its links are links of $X$.
 So, $\R^m\times L\times ]0,1[$ is a connected, normal CS set.
 From \lemref{lem:nor}, we deduce that
 $\R^m\times L\times ]0,1[\menos \Sigma=\R^m\times (L\menos \Sigma_{\cL})\times ]0,1[$
 is connected. 
 Any $\ov p$-allowable point of the CS set $\R^m\times L\times ]0,1[$ can be replaced by a regular one 
 (see \lemref{lem:cheminreg}) and any two regular points of $\R^m\times (L\menos \Sigma_{\cL})\times ]0,1[$
 are connected by a regular path. So,
 $\pi_{0}^{\ov p}(\R^m\times L\times ]0,1[,\cI_{\R^m}\times\cL\times \cI_{]0,1[})=\pi_{0}^{\ov p}(L,\cL)$
 consists of a single element, cf. \propref{prop:homotopyxR}.
  \end{proof}

 \section{Intersection fundamental group and CS-coarsenings}\label{sec:pi1coarsening}
 
 \begin{quote}
 We establish the existence of isomorphisms between the fundamental intersection groups of a CS-coarsening 
in two situations. The first  concerns the case without exceptional stratum 
 and is an extension of \cite[Theorem 5.6]{CST9}
 to general perversities. The second takes into account the existence of 
 some specific exceptional strata within Thom-Mather's spaces.
 \end{quote}

\begin{proposition}\label{prop:noexceppi1}
Let $(X,\cS,\cT,x_{0})$ be a pointed CS-coarsening  without exceptional stratum
and $\ov{p}$ be a perversity on $\cS$ with $\ov{p}\leq \ov t$. 
Then there is an isomorphism,
\begin{equation}\label{equa:pi1noexcep}
\pi_{1}^{\ov p}(X,\cS,x_{0})\xrightarrow{\cong}\pi_{1}^{\iota_{\star\ov p}}(X,\cT,x_{0}).
\end{equation}
\end{proposition}

 \begin{proof}
 The surjectivity is a direct consequence of \lemref{lem:pi1surj}. For the injectivity,  
 the proof done in \cite[Theorem 5.6]{CST9} for a codimensional perversity  adapts immediately. The only
 modification is the replacement of $\ov{p}(\codim\, T)$ by $\ov{p}(T)$.
 \end{proof}
 
 If there are  exceptional strata, the result cannot be true, as shows \cite[Example 5.3]{CST9}, recalled below.

\begin{example}\label{exam:henri}
Let $\Sigma^2P$ be the double suspension  of the Poincaré sphere. 
This suspension being homeomorphic to the sphere $S^5$, we get a CS-coarsening 
$(S^5,\cS,\cT)$, where $\cT$ is the trivial stratification with only one regular stratum. 
The stratification $\cS$ has one regular stratum and four exceptional strata $S_0,\,S_1,\,S_2,\,S_3$, 
with $\dim S_0 =\dim S_1= 0$ and $ \dim S_2 =\dim S_3= 1$. 
Let us define a $K$-perversity $\ov p$ on $\cS$  by $D\ov p=\ov 1$, so which  verifies $\ov p \leq \ov t$. 
By definition of  the pushforward perversity, we have  $\iota_\star \ov p =\ov 0$.
The computation made in \cite{CST9}  gives 
$\pi^{\ov p}_1(S^5,\cS) \ne \{1\}$ and $\pi_1^{\iota_\star\ov p} (S^5,\cT)=\{1\}$.

\end{example}

 In \propref{prop:pi1ThomM}, we  provide  a sufficient condition for keeping the isomorphism between intersection fundamental groups of a CS-coarsening
 in the presence of exceptional strata. 
 We do not need the whole structure of Thom-Mather stratified space, or smoothly stratified space of 
  \cite{BanaglCie}, only the existence of neighborhoods of singular strata without control data.

\medskip
A stratification on a CS set, $(X,\cS)$, is equivalent 
(\cite[Proposition 1.5]{CT1} or \cite[Appendix A.5]{LurieSecondBook} for a
global presentation of this point of view)
to the data of a surjective continuous map $p_{\cS}\colon X\to \cS$,
where $\cS$ is the poset of strata endowed with the Alexandrov topology
($U\subset \cS$ is open if $S\preceq S'$ and $S\in U$ implies $S'\in U$).

\medskip
Let $(X,\cS,\cT)$ be a CS-coarsening. We have a partition of $\cS$ in three parts,
formed of the  exceptional, fountain or source strata, respectively.
We define $\cS_{e}$ as the quotient of $\cS$ by the equivalence relation generated by $S\sim S^\iota$ if $S$ exceptional.
 This provides us with a new filtered space $p_{\cS_{e}}\colon X\to \cS_{e}$ where
 the singular strata of $\cS_{e}$ are the singular strata of $\cS$ which are not exceptional
 in $(X,\cS,\cT)$.
  A priori, given \exemref{exam:encadrement}, there is no reason why $(X,\cS_{e})$ should be a CS set. In fact,
 it is true and can be deduced from \cite[Proposition 3.10]{MR4170473}.
 We provide in \propref{prop:ST+} an independent simple proof in the  case of a pre-Thom-Mather space.

  \medskip
  We have three CS-coarsenings, $(X,\cS,\cT)$, $(X,\cS,\cS_{e})$, $(X,\cS_{e},\cT)$,
  which give a decomposition of the identity map in
 \begin{equation}\label{equa:decompoexception}
 \xymatrix@1{ 
 (X,\cS)\ar[r]^-{\iota_{e}}\ar@/_1pc/[rr]_{\iota}&
 (X,\cS_{e})\ar[r]^-{\iota_{\Sigma}}&
 (X,\cT).
 }
 \end{equation}
By construction, the following properties are satisfied.
 \begin{enumerate}[i)]
 \item A singular source stratum or a fountain stratum $S$ in $(X,\cS,\cT)$ is not modified in $(X,\cS_{e})$, so $S=S^{\iota_{e}}$. 
 Therefore $S$ becomes a source stratum
 in $(X,\cS,\cS_{e})$ and  the coarsening
 $(X,\cS,\cS_{e})$ has only exceptional and source strata,
 \item The CS-coarsening $(X,\cS_{e},\cT)$ has no exceptional strata.
  \end{enumerate}

 \begin{definition}\label{def:petitThomMather}
A  \emph{pre-Thom-Mather space} of dimension $n$ is a Hausdorff filtered space $(X,\cS)$
of dimension $n$,
whose $i$-dimensional strata are $i$-dimensional topological manifolds for each~$i$, and
such that, for each singular stratum $S\in\cS$, there exist 
an open neighborhood $E$ of $S$,  endowed with the induced filtration,
a continuous map $\Pi\colon E\to S$,
an open cover, $\gU$, of $S$, 
a filtered compact  space, $(L,\cL)$, of (formal) dimension $n-\dim S-1$,
and,   for each $U\in\gU$,  a stratified homeomorphism
 $$\varphi_{U}\colon (\Pi^{-1}(U),\cS)\to (U\times\rc L, \cI_{U}\times \rc\cL),$$
verifying
$\pr_{1}\circ \varphi_{U}=\Pi$ and $\varphi_{U}(x)=(x,\tv)$
for any $x\in U$, with $\tv$ the apex of the cone.
The map $\Pi$  is a \emph{conical neighborhood} of $S$ and
$\left\{\gU, (\varphi_{U})_{U\in\gU}\right\}$ is a  \emph{conical trivialization} of $\Pi$.
The filtered space $(L,\cL)$ is \emph{the link} of $S$.
\end{definition}

\begin{proposition}\label{prop:ST+}
Let $(X,\cS,\cT)$ be a  CS-coarsening  where $(X,\cS)$ is a pre-Thom-Mather space.
Then, in the decomposition \eqref{equa:decompoexception} of the identity map, 
the filtered space $(X,\cS_{e})$ is a pre-Thom-Mather space.
Moreover,  any singular stratum in $\cS_{e}$ is a singular stratum in $\cS$ and
one can choose the same conical neighborhood for $S$ in $(X,\cS)$ and in $(X,\cS_{e})$.
\end{proposition}

\begin{proof}
We first show that the association $(X,\cS,\cT)\mapsto (X,\cS_{e})$ is hereditary;
i.e., for any open subset $Y$ of $X$, one has $(\cS_{e})_{Y}=(\cS_{Y})_{e}$.
So let $Y$ be an open subset of $X$ and  $Q\in \cS_{Y}$. 
Then $Q$ is a connected component of $S\cap Y$ with $S\in \cS$.
As the restriction to $Y$ keeps the dimension of strata, we deduce that 
$Q$ is singular if and only if $S$ is, and that $Q$ is exceptional if and only if $S$ is.
This implies the heredity.

\medskip
Let $S\in\cS_{e}$ be a singular stratum. By construction, $S\in\cS$ is a singular source stratum of $(X,\cS,\cS_{e})$.
As $(X,\cS)$ is a pre-Thom-Mather space, there exist a 
conical neighborhood $\Pi\colon E\to S$ and  a {conical trivialization}
$\left\{\gU, (\varphi_{U})_{U\in\gU}\right\}$, with a stratified homeomorphism
$$\varphi_{U}\colon (\Pi^{-1}(U),\cS)\to (U\times\rc L, \cI_{U}\times \rc\cL).$$
As the association $(X,\cS,\cT)\mapsto (X,\cS_{e})$ is hereditary, we get a stratified homeomorphism
$$\varphi_{U}\colon (\Pi^{-1}(U),\cS_{e})\to (U\times\rc L, (\cI_{U}\times \rc\cL)_{e}).$$
We have to determine $(\cI_{U}\times \rc\cL)_{e}$.
As $\varphi_{U}(S)= U\times \{\tv\}$, we reduce the question to $\Pi^{-1}(U)\menos S$ and the study to
$(\cI_{U}\times \cL\times \cI_{]0,1[})_{e}$. 
The  construction of $\cS_{e}$ consists of a quotient in the poset $\cS$ (or an amalgamation of strata 
in $X$), thus  $(\cI_{U}\times \cL\times \cI_{]0,1[})_{e}$ is of the form 
$(\cI_{U}\times \cL'\times \cI_{]0,1[})$ where $\cL'$ is a coarsening of $\cL$.
This gives the result.
\end{proof}

\begin{remark}
If $(X,\cS)$ is a Thom-Mather space, it is (see \cite[Section 2.8.2]{LibroGreg} for details) a pseudomanifold. Therefore,
$(L,\cL)$ is a pseudomanifold. The space $L$ is also endowed with a stratification $\cL_{\cT}$ induced by $\cT$, giving a
coarsening $(L,\cL,\cL_{\cT})$.  One can observe that the coarsening $\cL'$ of the previous proof corresponds to
its associated $\cL_{e}$.
\end{remark}

In reference to the decomposition \eqref{equa:decompoexception} of the identity map, 
let us notice that, in the stratified map $\iota_{\Sigma}$, 
 the two stratifications have the same regular set.
 In \cite{CST9}, we have proven that  $\iota_{\Sigma}$  induces an isomorphism between intersection homotopy groups
 for GM-perversities.
 In \thmref{thm:coarseningnoexcep}, we extend it  to general perversities of MacPherson 
 (\cite{RobertSF}, \defref{def:perversity}).
 The  case $\iota_{e}$ was not considered in \cite{CST9}; we study it in \thmref{thm:coarseningwithexcep}.
 The next result is a step in this direction.

   \begin{proposition}\label{prop:pi1ThomM}
 Let $(X,\cS,\cT,x_{0})$ be a pointed CS-coarsening  where $(X,\cS)$ is a
 connected normal pre-Thom-Mather space  with a finite number of strata.
 Let $\ov p$ be a $K$-perversity  of pushforward perversity $\iota_{\star}\ov p$ and such that $\ov p\leq \ov t$.
If the cone on the links of all exceptional strata, $S$,
verifies $\pi_{1}^{\ov p}(\rc L,\rc\cL,x_{0})=\{1\}$, 
then the identity map induces  an isomorphism,
 \begin{equation}\label{equa:pi1ThomM}
\pi_{1}^{\ov p}(X,\cS,x_{0})\xrightarrow{\cong}\pi_{1}^{\iota_\star\ov p}(X,\cT,x_{0}).
\end{equation}
 \end{proposition}
 
 \begin{proof}
Let us consider the decomposition recalled in \eqref{equa:decompoexception},
\begin{equation}
 \xymatrix@1{
 (X,\cS)\ar[r]^-{\iota_{e}}&
 (X,\cS_{e})\ar[r]^-{\iota_{\Sigma}}&
 (X,\cT).
 }
 \end{equation}
 By \propref{prop:noexceppi1} and \lemref{DI},  
 the simplicial map $\iota_{\Sigma}$ induces an isomorphism between the
 intersection fundamental groups. So we are reduced to the study of the CS-coarsening $(X,\cS,\cS_{e})$
 induced by $(X,\cS,\cT)$.

 \medskip
 The proof is done  by induction on the depth and the number of
minimal singular strata in $\cS$. 
Let $S\in\cS$ be a minimal stratum with a 
conical neighborhood $\Pi\colon E\to S$ of associated link $(L,\cL_{\cS})$.
The inductive step  consists of  an open cover of $X$ by $E$ and $X\menos S$. 
Let us observe that $E\cap (X\menos S)=E\menos S$. The induction hypothesis can be applied to
$X\menos S$ and $E\menos S$. Therefore, with
the Van Kampen theorem of intersection fundamental groups (\cite[Theorem 5.1]{CST9}),
the proof is reduced to the case of $E$ with the induced filtrations.
The hypotheses of connectivity, required in \cite[Theorem 5.1]{CST9}, are satisfied: 
for $E$, it is a consequence of the connectivity of $S$ and $\rc L$, for $X\menos S$ it comes from
$X\menos \Sigma_{\cS}\subset X\menos S$ and $\ov{X\menos \Sigma_{\cS}}=X$.

  \medskip
By \cite[Theorem 2.3]{Gajer1}, the  map $\Pi$ is a locally trivial filtered fibration 
from which we deduce (\cite[Theorem 2.2]{Gajer1}) a   long exact sequence in intersection
 homotopy groups. 
 (When the spaces are  path connected, we  omit the reference to a basepoint.)
We already have  informations on some groups in this sequence.
 First, as there is only one stratum in $S$, we have $\pi_{k}^{\ov q}(S,\cI_{S})=\pi_{k}(S)$ for any perversity $\ov q$.
We know from Propositions \ref{prop:normal} and \ref{prop:connectionLink} that 
$\pi_{0}^{\ov p}(L,\cL_{S})=\pi_{0}^{\ov p}(L,\cL_{S_{e}})$ consists of a single element.
The cone formula \eqref{equa:cone} being also verified for $\pi_{0}^{\ov p}$, we get
$\pi_{0}^{\ov p}(\rc L,\rc\cL_{S})=\pi_{0}^{\ov p}(\rc L,\rc\cL_{S_{e}})=\{\star\}$.
We have two possibilities for   the singular stratum $S\in\cS$.

 \smallskip
 \emph{Suppose that $S$ is exceptional in $(X,\cS,\cS_{e})$.} 
 By hypothesis, we have $\pi_{1}^{\ov p}(\rc L,\rc\cL_{\cS})=\{1\}$.
Here, the total space of the locally trivial fibration
 $\rc L\to E\to S$
 has only one stratum in $\cS_{e}$; thus the corresponding long exact sequence of intersection homotopy groups 
 is the long exact sequence of homotopy groups. So we have a diagram of short exact sequences,
$$\xymatrix{
\pi_{2}(S)\ar[r]\ar[d]^\cong&\pi_{1}^{\ov p}(\rc L,\rc\cL_{\cS})\ar[r]\ar[d]^\cong&\pi_{1}^{\ov p}(E,\cS)\ar[r]\ar[d]&\pi_{1}(S)\ar[r]\ar[d]^\cong&0\\
\pi_{2}(S)\ar[r]&\pi_{1}(\rc L)\ar[r]&\pi_{1}(E)\ar[r]&\pi_{1}(S)\ar[r]&0,
}$$
and an isomorphism
$\pi_{1}^{\ov p}(E,\cS)\cong \pi_{1}(E)$. 

\medskip
 \emph{Suppose that $S$ is not exceptional in $(X,\cS,\cS_{e})$.} Then $S$ is singular in $(X,\cS_{e})$ and, by 
\propref{prop:ST+}, we can choose $\Pi$ as conical neighborhood for $\cS$ and $\cS_{e}$. 
 We thus have a diagram of exact sequences:
 $$\xymatrix{
\pi_{2}(S)\ar[r]\ar[d]^\cong&\pi_{1}^{\ov p}(\rc L,\rc\cL_{\cS})\ar[r]\ar[d]^{\psi_{\rc L}}&\pi_{1}^{\ov p}(E,\cS)\ar[r]\ar[d]^{\psi_{E}}&\pi_{1}(S)\ar[r]\ar[d]^\cong&0\\
\pi_{2}(S)\ar[r]&\pi_{1}^{\iota_{\star}\ov p}(\rc L,\rc\cL_{\cS_{e}})\ar[r]
&\pi_{1}^{\iota_{\star}\ov p}(E,\cS_{e})\ar[r]&\pi_{1}(S)\ar[r]&0.
}$$
 \emph{Claim: The map $\psi_{\rc L}$ is an isomorphism,} from which we deduce that 
 $\psi_{E}\colon \pi_{1}^{\ov p}(E,\cS)\xrightarrow{\cong}\pi_{1}^{\iota_{\star}\ov p}(E,\cS_{e})$ is an isomorphism,
 as required.

 \medskip
 From \defref{def:petitThomMather} and \propref{prop:homotopyxR}, we have to prove that the homomorphism
 $$\iota_{\Pi^{-1}(U)}\colon \pi_{1}^{\ov p}(\Pi^{-1}(U),\cS)\to
 \pi_{1}^{\iota_{\star}\ov p}(\Pi^{-1}(U),\cS_{e})
 $$
 is an isomorphism. Since $S^{\iota_{e}}=S$, the stratum $S$ is a source and we have
 $D\ov p(S)=\iota_{\star}D\ov p(S)=D \iota_{\star}\ov p(S)$, cf. \lemref{LemTec}. 
 We use now the cone formula \eqref{equa:cone}.
 If $D\ov p(S)=0$, the domain and codomain of $\iota_{\Pi^{-1}(U)}$ are equal to $\{1\}$. 
 If not, we have $\pi_{1}^{\ov p}(\Pi^{-1}(U),\cS)\cong \pi_{1}^{\ov p}(\Pi^{-1}(U)\menos S,\cS)$
 and a similar isomorphism with the stratification $\cS_{e}$.
 The result comes from the induction applied to
 $$\iota_{\Pi^{-1}(U)\menos S}\colon \pi_{1}^{\ov p}(\Pi^{-1}(U)\menos S,\cS)\to
 \pi_{1}^{\iota_{\star}\ov p}(\Pi^{-1}(U)\menos S,\cS_{e}).
 $$
 \end{proof}

 \begin{example}\label{ex:jpb}
 From \lemref{lem:pi1surj} and the cone formula \eqref{equa:cone}, 
 we know that the hypothesis $\pi_{1}^{\ov p}(\rc L,\rc\cS)=\{1\}$ in \propref{prop:pi1ThomM} is a consequence of 
 $\pi_{1}(L\menos \Sigma_{L})=\{1\}$. Yet, this last condition is stronger, as shows the following example.

Let $L$ be the pinched torus described in \cite{GM1}. 
Its stratification $\cL$ has exactly one singular point and verifies  $\pi_{1}(L\menos \Sigma_{L}) = \Z$. 
We choose  $\ov p=\ov 0$.
As we have already pointed out, the allowable 1-simplices are the simplices included in the regular part
and an allowable 2-simplex can meet the singular set in dimension~0.
So,  the generator of $\pi_{1}(L\menos \Sigma_{L})$ is killed in $\pi_{1}^{\ov 0}(L,\cL)$.
  In short,   we have 
  $ \pi_1(L\menos \Sigma_L) =\Z$ 
  and  
  $\pi_{1}^{\ov{0}}(\rc L,\rc\cL)=
  \{1\}$.
\end{example}
 
  \section{Proofs of Theorems~\ref{thm:coarseningnoexcep} and \ref{thm:coarseningwithexcep}}\label{sec:proof}
 
 \begin{quote}
 Let $(X,\cS,\cT)$ be a CS-coarsening  endowed  with a $K$-perversity $\ov p$.
 The identity map induces a simplicial map 
 $\crG(\iota)\colon \crG_{\ov p}(X,\cS)\to \crG_{\iota_{\star}\ov p}(X,\cT)$.
 In this section, we prove that the latter  induces  a homology isomorphism
 $H_{\star}(\crG_{p}(X,\cS);\crG(\iota)^\star\gE)
 \xrightarrow{\cong}
  H_{\star}(\crG_{\iota_{\star}\ov p}(X,\cT);\gE)$, for any local coefficient system, $\gE$, on 
 $\crG_{\iota_{\star}\ov p}(X,\cT)$. 
 We then deduce the proofs of Theorems~\ref{thm:coarseningnoexcep} and \ref{thm:coarseningwithexcep}.
 \end{quote}

To achieve this objective,  recall  a slight modification of \cite[Theorem 5.1.4]{LibroGreg}, see also \cite[Theorem 5.1]{CST3}.

\begin{proposition}\label{prop:Met1}
Let $\mathcal F_{X}$ be the category whose objects are (stratified homeomorphic to) open subsets
of a given  CS set $(X,\SS)$ and whose morphisms are  stratified homeomorphisms and inclusions.
Let  $\mathcal Ab_{\star}$ be the category of graded abelian groups. Let $F_{\star},\,G_{\star}\colon \mathcal F_{X}\to \mathcal Ab$
be two functors and
 $\Phi\colon F_{\star}\to G_{\star}$ be a natural transformation satisfying
 the conditions listed
below.
\begin{enumerate}[(i)]
\item The functors $F_\star$ and $G_{\star}$ admit exact Mayer-Vietoris sequences and the natural transformation $\Phi$ 
 induces a commutative diagram between these sequences.
\item If $\{U_{\alpha}\}$ is a increasing collection of open subsets of $X$  and $\Phi\colon F_{\star}(U_{\alpha})\to G_{\star}(U_{\alpha})$ is an isomorphism for each $\alpha$, then $\Phi\colon F_{\star}(\cup_{\alpha}U_{\alpha})\to G_{\star}(\cup_{\alpha}U_{\alpha})$  is an isomorphism.
\item Let $(\varphi,V)$ be any element of a basis of  conical charts of a singular point $x \in S$ with $S \in \SS$.
If 
$\Phi\colon F_{\star}(V\menos S)\to G_{\star}(V\menos S)$
is an isomorphism, then so is
$\Phi\colon F_{\star}(V)\to G_{\star}(V)$. 
\item If $U$ is an open subset of X contained within a single stratum and homeomorphic
to  an Euclidean space, then $\Phi\colon F_{\star}(U)\to G_{\star}(U)$ is an isomorphism.
\end{enumerate}
Then $\Phi\colon F_{\star}(X)\to G_{\star}(X)$ is an isomorphism.
\end{proposition}

\begin{proof}
The only modification compared to the statement in Friedman's book lies in item (iii). 
In our case, property (iii) is required only for  a basis of conical charts.
 A careful reading of the proof in \cite[Theorem 5.1.4]{LibroGreg}  shows that this restriction is sufficient.
\end{proof}

Let $(X,\cS,\cT)$ be a  CS-coarsening.
According to  \eqref{equa:gajermap},  the identity map of $X$ defines a simplicial map
$\crG (\iota)\colon \crG_{\ov p}(X,\cS)\to \crG_{\iota_{\star}\ov p}(X,\cT)$.
A local coefficient system $\gE$ on $\crG_{\iota_{\star}\ov p}(X,\cT)$ induces local coefficient systems,
denoted $\iota^\star\gE$,
obtained by pullback of $\gE$,  on $\crG_{\ov p}(X,\cS)$ or on simplicial subsets. 
We should also note that each open subset $U$ of $X$ owns two structures of CS set that we denote
by $(U,\cS)$ and $(U,\cT)$.

\begin{corollary}\label{cor:kingargument}
Let $(X,\cS,\cT)$ be a  CS-coarsening,
$\ov p$ be a $K$-perversity  of pushforward perversity $\iota_{\star}\ov p$ on $(X,\cT)$
and $\gE$ be a local coefficient system on $\crG_{\iota_\star \ov p}( X, \cT)$.
Suppose that the following property is satisfied 
for a basis of  conical charts $(\varphi,V,\cS)$ in the CS set $(X,\cS)$, 
of a singular point $x\in S$ with $S\in\cS$:\\
``if $
\crG(\iota)_\star \colon H_\star(\crG_{\ov p}(V\menos S,  \cS);\iota^\star\gE)) 
\xrightarrow \cong
 H_\star(\crG_{\iota_\star \ov p}(V\menos S,  \cT),\gE)$
 is an isomorphism, then so is \\
$
\crG(\iota)_\star \colon H_\star(\crG_{\ov p}(V,\cS);\iota^\star\gE)
\xrightarrow \cong 
H_\star(\crG_{\iota_\star \ov p}(V,\cT);\gE)$.''
Then the identity map induces an isomorphism
$$
\crG(\iota)_\star \colon H_\star(\crG_{\ov p}( X, \cS);\iota^\star\gE )  
\xrightarrow{\cong}
 H_\star(\crG_{\iota_\star \ov p}( X, \cT);\gE ). 
$$
\end{corollary}

\begin{proof}
The proof consists of the verifications of the hypotheses of \propref{prop:Met1} for 
the CS set $(X,\cS)$,
$F_{\star}(-)=H_{\star}(\crG_{\ov p}(-,\cS);\iota^\star\gE)$
and $G_{*}(-)=H_{\star}(\crG_{\ov p}(-,\cT);\gE)$.
For  (ii) and (iv), it is immediate. Condition (i)
comes from the existence of a Mayer-Vietoris sequence in \cite[Theorem 3.13]{CST9}.
Finally, condition (iii) is required as a hypothesis in the statement .
\end{proof}

This corollary is  the appropriate tool to demonstrate the existence of  isomorphisms between homologies 
as  below in Theorems \ref{thm:Coars} and \ref{thm:CoarsThom}. 
In the proof of these theorems, we also use a combination of
homotopy and homology computations, where the following lemma will be essential.
Weak equivalences can be detected from connected components, fundamental groups and isomorphisms of homology groups for any local coefficient system (\cite[Proposition II.3.4]{MR0223432}). We use the next slight modification.

\begin{lemma}[{\cite[Lemma 3.8]{CST9}}]\label{lem:quillenfini} 
Let $f\colon K\to K'$ be a simplicial map between Kan simplicial sets whose induced maps verify the following properties:
$\pi_{0}(K)\cong \pi_{0}(K')$, $\pi_{1}(K,x)\cong \pi_{1}(K',f(x))$ for any $x\in K_{0}$,
$H_{j}(K;f^\star\gE)\cong H_{j}(K';\gE)$, for 
any local coefficient system $\gE$ on $K'$, 
 any $j\leq \ell$ and
$H_{\ell+1}(K;f^\star\gE)\to H_{\ell+1}(K';\gE)$ is a surjection.
Then, the map induced by $f$ between the  homotopy groups is an isomorphism 
for $j\leq \ell$ and a surjection for $j=\ell+1$.
\end{lemma}

\begin{theoremb}\label{thm:Coars}
 Let $(X,\cS,\cT)$ be a  CS-coarsening   without exceptional strata, 
  and let $\ov p$ be a $K$-perversity with   $\ov p\leq \ov t$, of pushforward perversity $\iota_{\star}\ov p$.
  Then, for 
any local coefficient system $\gE$ on $\crG_{\iota_\star \ov p}( X, \cT)$, the identity map induces an isomorphism
$$
\crG(\iota)_\star \colon H_\star(\crG_{\ov p}( X, \cS);\iota^\star\gE )  
\xrightarrow{\cong}
 H_\star(\crG_{\iota_\star \ov p}( X, \cT);\gE ). 
$$
 \end{theoremb}

\begin{proof}
From \cite[Proposition 2.3]{MR4742271}, the identity map
$(X,\cS)\to (X,\cT)$
can be decomposed in a sequence of simple CS-coarsenings (\defref{def:simple}).
So, with \lemref{DI}, we can reduce the problem to a simple CS-coarsening where
the two conical charts  are connected as detailed after the loc. cit. definition and recalled below.
The proof consists of the verifications of the hypotheses of \corref{cor:kingargument}.
The different systems of coefficients that appear are obtained by pullback of $\gE$;
 we do not mention them in the rest of the proof.

\smallskip
Let $x\in S$ be a  point in a singular stratum $S\in \cS$. We consider a $\cS$-conical chart containing $x$ as in \secref{sec:coarsening}.
We distinguish two cases. 

\medskip\noindent
$\bullet$ \emph{The stratum $S$ is a source stratum,} as described in  Diagram \eqref{equa:source}
and recalled for the convenience of the reader,
\begin{equation}\label{equa:source2}
\xymatrix{
(\R^a \times \rc L, \I_{\R^a} \times  \rc \mathcal L)  \ar[d]_{\iota} \ar[r]^-\psi  & (W,\mathcal S) \ar[d]^\iota\ \\
(\R^a \times \rc L, \I_{\R^a} \times \rc \mathcal L') \ar[r]^-\psi & (W,\mathcal T).
}
\end{equation}
Using \propref{prop:homotopyxR} and \corref{cor:kingargument}, we have to prove:\\
\emph{``If $
\crG(\iota)_\star \colon H_\star(\crG_{\ov p}( L, \cL)) 
\xrightarrow \cong
 H_\star(\crG_{\iota_\star \ov p}(L,  \mathcal L'))$
is an isomorphism, then so is \\
$
\crG(\iota)_\star \colon H_\star(\crG_{\ov p}( \rc L, \rc\cL))
\xrightarrow \cong 
H_\star(\crG_{\iota_\star \ov p}(\rc L,  \rc\mathcal L'))$.''}

\smallskip
Recall from  \lemref{LemTec} the equalities
  $D\iota_\star \ov p(\tv)= \iota_\star D\ov p(\tv)  = D \ov p(\tv) $, where $\tv$ is the apex of the cone $\rc L$.
From \eqref{equa:homologyGajercone}, we deduce
  isomorphisms and a surjection as follows:
$$
\begin{array}{ll}
\iota_\star \colon H_{k } (\crG_{\ov p}( \rc L,\rc \mathcal L) ) 
\xrightarrow{\cong} 
H_{k} (\crG_{\iota_\star \ov p}(\rc L,\rc \mathcal L')),
&
\text{if } k  \leq D \ov p (\tv), \text{ and} \\
\iota_\star \colon H_{k} (\crG_{\ov p}( \rc L,\rc \mathcal L) ) 
\twoheadrightarrow 
H_{k} (\crG_{\iota_\star \ov p}(\rc L,\rc \mathcal L')),
&
\text{if } k  =  D\ov p (\tv)+1.
\end{array}
$$
Let $x_{0}$ be any regular basepoint.
As the conical chart has no exceptional stratum,
  \lemref{lem:quillenfini}, Propositions \ref{lemapi1} and \ref{prop:noexceppi1} imply  isomorphisms and a surjection as
 $$
\begin{array}{ll}
\iota_\star \colon  \pi_{k}^{\ov p}( \rc L,\rc \mathcal L,x_0) 
\xrightarrow{\cong}
 \pi_{k}^{\iota_\star \ov p}(\rc L,\rc \mathcal L',x_0),
&
\text{if } k  \leq D\ov p (\tv), \text{ and}\\
\iota_\star \colon \pi_{k }^{\ov p}( \rc L,\rc \mathcal L ,x_0) 
\twoheadrightarrow 
\pi_{k}^{\iota_\star \ov p}(\rc L,\rc \mathcal L',x_0), 
&
\text{if } k  =  D \ov p (\tv) + 1.
\end{array}
$$
On the other hand, we know from \eqref{equa:cone} that 
$\pi_{\ell }^{\ov p}( \rc L,\rc \mathcal L,x_0) 
= 
\pi_{\ell}^{\iota_\star \ov p}(\rc L,\rc \mathcal L',x_0) =0$, if $\ell > D\ov p(\tv)$.
 We conclude that the map
$$
\iota_\star \colon  \pi_{\star}^{\ov p}( \rc L,\rc \mathcal L,x_0) 
\xrightarrow{\cong} 
\pi_{\star}^{\iota_\star \ov p}(\rc L,\rc \mathcal L',x_0) 
$$
is an isomorphism. So the classical Whitehead theorem gives the claim:
$$
\iota_\star \colon H_{\star} (\crG_{\ov p}( \rc L,\rc \mathcal L) ) 
\xrightarrow{\cong} 
H_{\star} (\crG_{\iota_\star \ov p}(\rc L,\rc \mathcal L')).
$$

\medskip\noindent
$\bullet$ \emph{The stratum $S$ is a fountain stratum}  
as described in Diagram \eqref{equa:virtual2} 
which we reproduce here:
 \begin{equation}\label{equa:virtual3}
\xymatrix{
(\R^a \times  (\rc \S^{b-1} \times \rc L), \cI_{\R^a} \times (\cI_{\rc\S^{b-1}} \times \rc \cL)_{(\tu,\tv)}) \ar[d]_-{\iota} \ar[r]^-{\phi'}  
&
 (W,\cS) \ar[d]^-{\iota}\ \\
((\R^a \times \rc \S^{b-1} )\times  \rc L, \cI_{\R^a\times \rc\S^{b-1}}  \times \rc \mathcal L) \ar[r]^-{\psi'} & (W,\cT).
}
\end{equation}
(Recall that $\tu$ is the apex of $\rc \S^{b-1}$ and $\tv$ that of $\rc L$.)
With \corref{cor:kingargument} and the involved perversities  described in \remref{rem:perversityonlink}, 
we have to prove:\\
\emph{ ``The  map 
 \begin{equation}\label{equa:virtualandhomology}
\iota \colon A:=(\R^a\times \rc \S^{b-1} \times \rc L, (\I_{\R^a\times \rc\S^{b-1}} \times \rc \cL)_{(\tu,\tv)}) \to 
B:=(\R^a\times \rc \S^{b-1} \times \rc L, \I_{\R^a\times \rc \S^{b-1}} \times \rc \cL)
\end{equation}
induces an isomorphism
 $\iota_\star \colon H_{\star }(\crG_{\ov p} A) 
\longrightarrow 
H_{\star} (\crG_{\iota_\star \ov p}B)$.'' }

\smallskip\noindent
We  now use a version of the cone formula 
\eqref{equa:homologyGajercone} 
for a product of two Gajer spaces on cones, its proof being postponed to \lemref{lem:CalJoin}. 
So, there are isomorphisms and a surjection as follows: 
\begin{equation}\label{equa:thetwocones}
\begin{array}{ll}
\iota_\star \colon H_{k }(\crG_{\ov p} A) 
\xrightarrow{\cong} 
H_{k} (\crG_{\iota_\star \ov p}B ),
&
\text{if } k \leq D\ov{p}(\{ \tu,\tv\} ), \text{ and}
\\
\iota_\star \colon H_{k }(\crG_{\ov p}A) 
 \twoheadrightarrow 
H_{k} (\crG_{\iota_\star \ov p}B), 
&
\text{if } k = D\ov{p}(\{ \tu,\tv\} )+1.
\end{array}
\end{equation}
 As the conical chart has no exceptional stratum, with \lemref{lem:quillenfini}, 
\propref{lemapi1} and \propref{prop:noexceppi1}, we also conclude with isomorphisms and a surjection: 
$$
\begin{array}{ll}
\iota_\star \colon \pi_{k }^{\ov p}( A,x_0) \xrightarrow \cong  
\pi_{k}^{\iota_\star \ov p}(B,x_0),
&
\text{if } k \leq D\ov{p}(\{ \tu,\tv\} ), \text{ and}
\\
\iota_\star \colon \pi_{k }^{\ov p}(A,x_0)
 \twoheadrightarrow 
\pi_{k}^{\iota_\star \ov p}(B,x_0),
&
\text{if } k = D\ov{p}(\{ \tu,\tv\} )+1,
\end{array}
$$
for any regular basepoint $x_0\in \R^a\times \rc \S^{b-1} \times (L\menos \Sigma_{L}) \times ]0,1[$.
On the other hand, we have  from \eqref{equa:cone}: 
$\pi_{\ell }^{\ov p}( A, x_0) = 0$ if $\ell > D\ov{p}(\{ \tu,\tv\} )$ 
and 
$\pi_{\ell}^{ \iota_{\star}\ov p}(B,x_0) = 0$ if $\ell > D\ov{p}(\{ \tv\} )$. 
So, since $D\ov{p}(\tv) \leq D\ov{p}(\{ \tu,\tv\} )$ (cf. \remref{rem:perversityonlink}), we obtain 
$\pi_{\ell }^{\ov p}( A,x_0) \cong \pi_{\ell}^{ \iota_{\star}\ov p}(B,x_0) =0$ if $\ell >  D\ov{p}(\{ \tu,\tv\} )$
and the map
$$
\iota_\star \colon  \pi_{\star}^{\ov p}( A,x_0) \xrightarrow \cong \pi_{\star}^{\iota_{\star}\ov p}(B,x_0) 
$$
is an isomorphism for any regular basepoint $x_0$. Finally,
the classical Whitehead theorem gives the claim.
\end{proof}

We prove the formula for two cones \eqref{equa:thetwocones} used in the proof of \thmref{thm:Coars}.

\begin{lemma} \label{lem:CalJoin}
Let $(L,\cL)$ be a compact stratified space and $b>0$ be an integer. 
We denote by $\tu$ the apex of $\rc \S^{b-1}$ and by $\tv$ that of $\rc L$.
Then, for any $K$-perversity  $\ov p$  on $A$, the  stratified map defined by the identity,
$$
\iota \colon 
A:=(\R^a\times\rc \S^{b-1} \times \rc L, \cI_{\R^a}\times (\I_{\rc \S^{b-1}} \times \rc \cL)_{(\tu,\tv)})
\to 
B:=(\R^a\times\rc \S^{b-1} \times \rc L,  \I_{\R^a\times \rc\S^{b-1}} \times \rc \cL),
$$
induces
 $$
 \left\{
\begin{array}{lll}
\hbox{isomorphisms} &
\iota_{\star} \colon H_{k }(\crG_{\ov p} A) 
  \xrightarrow{\cong} 
H_{k} (\crG_{\iota_\star \ov p}B)
&
\hbox{if } k \leq D\ov{p}(\{ \tu,\tv\} ), \hbox{ and}
\\
\hbox{an epimorphism} & 
\iota_{\star} \colon H_{k }(\crG_{\ov p}A) 
 \twoheadrightarrow 
H_{k} (\crG_{\iota_\star \ov p}B) 
&
\hbox{if } k = D\ov{p}(\{ \tu,\tv\} )+1.
\end{array}
\right.
$$
\end{lemma}

\begin{proof}
For  sake of simplicity, we set $\ov q=\iota_\star \ov p$.
From \eqref{equa:gajermap}, we know that $\iota$ induces a homomorphism
$\iota_{\star} \colon H_{\star}(\crG_{\ov p}A) \to H_{\star}(\crG_{\ov q}B)$.
We consider the two stratified sets
\begin{eqnarray*}
C&:=&  (\R^a\times(\rc \S^{b-1} \times \rc L )\menos (\R^a\times \{(\tu,\tv)\}), \I_{\R^a}\times 
(\I_{\rc \S^{b-1}} \times \rc \cL)_{(\tu,\tv)}) \\
&=&
  (\R^a\times(\rc \S^{b-1} \times \rc L) \menos (\R^a\times \{(\tu,\tv)\}, \I_{\R^a\times \rc\S^{b-1}} \times \rc \cL)
\end{eqnarray*}
and
\begin{eqnarray*}
D&:=&  (\R^a\times\rc \S^{b-1} \times (\rc L \menos \{\tv\}), \I_{\R^a}\times ( \I_{\rc\S^{b-1}} \times \rc \cL)_{(\tu,\tv)}) \\
&=&
   (\R^a\times\rc \S^{b-1} \times (\rc L \menos \{\tv\}), \I_{\R^a\times\rc\S^{b-1}} \times \rc \cL).
\end{eqnarray*}
The first step of the proof is the determination of $H_{{\star}}(\crG_{\ov{q}}B,\crG_{\ov{q}}C)$. 
To do this, we first use the excision theorem (\cite[Theorem 3.14]{CST9}) of  
$Z = \R^a\times\{ \tu\} \times (\rc L \menos \{\tv\})$ relatively to the pair $(C,D)$. 
So the classical {relative} homology verifies:
\begin{eqnarray*}
H_{k}(\crG_{\ov{q}}C,\crG_{\ov{q}}D)
&\cong & 
H_{k}(\crG_{\ov{q}}(C\menos Z) ,\crG_{\ov{q}}(D\menos Z))\\
&\cong&  
H_{k}(\crG_{\ov{q}}(\R^a\times (\rc \S^{b-1} \menos \{\tu\} ) \times \rc L ) , \crG_{\ov{q}}( \R^a\times (\rc \S^{b-1} \menos \{\tu\} ) \times (\rc L \menos \{ \tv\}) ) ).
\end{eqnarray*}
The stratified space $(\R^a\times \rc \S^{b-1}\menos \{\tu\},\cI_{\R^a\times \rc \S^{b-1}\menos \{\tu\}})$ 
is stratified homeomorphic to 
$\R^a\times\S^{b-1}\times ]0,1[$ with the trivial filtration. From \propref{prop:homotopyxR}, we get
\begin{eqnarray} 
H_{k}(\crG_{\ov{q}}C,\crG_{\ov{q}}D) 
&\cong&
\nonumber
H_{k}(\Sing\, \S^{b-1}\times  \crG_{\ov{q}}( \rc L ) , 
\Sing\, \S^{b-1}\times \crG_{\ov{q}}( \rc L \menos \{ \tv\})  )\\ \nonumber
&\cong&
H_{k}(\crG_{\ov{q}}( \rc L ), \crG_{\ov{q}}( \rc L \menos \{ \tv\})  )
\oplus
H_{k-b+1}(\crG_{\ov{q}}(  \rc L ) , \crG_{\ov{q}}(\rc L \menos \{ \tv\}) )\\
&\cong&
H_{k}(\crG_{\ov{q}}B, \crG_{\ov{q}}D  )
\oplus
H_{k-b+1}(\crG_{\ov{q}}B, \crG_{\ov{q}}D  ).\label{equa:excision}
\end{eqnarray}
Thus, for each $k\in \N$, the  inclusion $ C \hookrightarrow B$ induces an epimorphism
$\nu_k \colon H_{k}(\crG_{\ov{q}}C, \crG_{\ov{q}}D  ) \twoheadrightarrow 
H_{k}(\crG_{\ov{q}}B,\crG_{\ov{q}}D)
$.
So, for \emph{any} $k$, the long exact  sequence of the triple $(\crG_{\ov{q}}B,\crG_{\ov{q}}C,\crG_{\ov{q}}D)$,
breaks down in short exacts sequences,
\begin{equation}\label{equa:shortandtop}
0
\to
H_{k+1}(\crG_{\ov{q}}B, \crG_{\ov{q}}C)
\to
H_{k}(\crG_{\ov{q}}C,\crG_{\ov{q}}D)
\xrightarrow{\nu_{k}}
H_{k}(\crG_{\ov{q}}B,\crG_{\ov{q}}D)
\to
0.
\end{equation}
Given $k\leq D\ov p(\{\tu,\tv\})$, we deduce from \eqref{eq:AB}:
$$k-b+1\leq D\ov p(\{\tu,\tv\})-b+1\leq D\ov q(\tv)+b-b+1=D\ov q(\{\tv\})+1.$$
Thus,  the cone formula \eqref{equa:conegajer} and the isomorphism \eqref{equa:excision} imply that $\nu_{k}$
is an isomorphism and therefore,
 we have for the  classical relative homology:
\begin{equation}\label{equa:relatifexcision}
H_{k}(\crG_{\ov{q}}B, \crG_{\ov{q}}C)=0, \text{ if } k\leq D\ov q(\{\tu,\tv\}) +1.
\end{equation}
The identity map $\iota\colon A\to B$ induces morphisms 
 between the homology long exact sequences of
the pairs $(\crG_{\ov{p}}A,\crG_{\ov{p}}C)$ and $(\crG_{\ov{q}}B,\crG_{\ov{q}}C)$. We
denote them by $\iota_{\star}^i$, for $i=1,2,3$, as follows:
$$\xymatrix{
\dots\ar[r]
&
H_{k+1}(\crG_{\ov{p}}A,\crG_{\ov{p}}C)\ar[r]\ar[d]^-{\iota^3_{k+1}}
&
H_{k}(\crG_{\ov{p}}C)\ar[r]\ar[d]^-{\iota^1_{k}}
&
H_{k}(\crG_{\ov{p}}A)\ar[r]\ar[d]^-{\iota^2_{k}}
&
H_{k}(\crG_{\ov{p}}A,\crG_{\ov{p}}C)\ar[d]^-{\iota^3_{k}}\ar[r]
&\dots \\
\dots\ar[r]
&
H_{k+1}(\crG_{\ov{q}}B,\crG_{\ov{q}}C)\ar[r]
&
H_{k}(\crG_{\ov{q}}C)\ar[r]
&
H_{k}(\crG_{\ov{q}}B)\ar[r]
&
H_{k}(\crG_{\ov{q}}B,\crG_{\ov{q}}C)\ar[r]
& \dots
	}
	$$
	Recall from \remref{rem:perversityonlink} that the two filtrations coincide on
	$\R^a\times (\tc \S^{b-1}\times \rc L)\menos (\R^a\times \{(\tu,\tv)\})$.
	So the two perversities also coincide on it and the
	map $\iota^1_k$ is an isomorphism for each $k$.
	Applying \eqref{equa:conegajer} together 
	with  the stratified homeomorphism $h$  \eqref{equa:hoeoconejoin}, we get
	$
	H_{k}(\crG_{\ov{p}}A,\crG_{\ov{p}}C)=0
	$
	for each $k \leq D\ov p(\{ \tu , \tv\}) + 1$.
So, with the 5-lemma and \eqref{equa:relatifexcision}, we conclude that 
$\iota^2_k \colon H_k(\crG_{\ov p} A) \to H_k(\crG_{\ov p} B)$ is an isomorphism for $k \leq D\ov p(\{\tu ,\tv\})$ 
	and an epimorphism for $k = D\ov p(\{\tu ,\tv\})+1 $.
\end{proof}

\begin{theoremb}\label{thm:CoarsThom}
Let $(X,\cS,\cT)$ be a CS-coarsening.
Let  $\ov{p}$ be a  $K$-perversity of pushforward perversity $\iota_{\star}\ov p$, such that $\ov p\leq \ov t$ and 
$\pi_{1}^{\ov p}(\rc L,\rc\cS,x_{0})=\{1\}$
for the  links $L$ of every exceptional stratum,  with $x_{0}$ regular.
Then, for 
any local coefficient system $\gE$ on $\crG_{\iota_\star \ov p}( X, \cT)$, the identity map induces an isomorphism
$$
\crG(\iota)_\star \colon H_\star(\crG_{\ov p}( X, \cS);\iota^\star\gE )  
\xrightarrow{\cong}
 H_\star(\crG_{\iota_\star \ov p}( X, \cT);\gE ). 
$$
 \end{theoremb}
 \begin{proof}
 We follow the process of the proof of \thmref{thm:Coars} using the decomposition in simple coarsenings
 established in \cite[Proposition 2.3]{MR4742271}. 
 In particular,  the proof reduces to an application of \corref{cor:kingargument}. 
Let $x\in S$ be a  point in a singular stratum $S\in \cS$. 
We may consider a small enough $\cS$-conical charts containing $x$ and suppose 
that the stratum $S$ is exceptional in $(X,\cS,\cT)$, since the two other cases 
were dealt with in  the proof of \thmref{thm:Coars}. 
Let us recall the local situation \eqref{equa:exceptional} in the diagram:
\begin{equation}\label{equa:exceptional2}
\xymatrix{
(\R^a \times \rc \S^{b-1}, \I_{\R^a} \times \rc \mathcal I_{\S^{b-1}})  \ar[d]_{f}\ar[r]^-\phi  & (W,\mathcal S) \ar[d]^\iota\ \\
(\R^{a+b}, \I_{\R^{a+b}} ) \ar[r]^-\psi & (W,\mathcal T).
}
\end{equation}
Notice that $b\geq 2$ since $\ov p$ is a $K$-perversity (cf. \remref{rem:propertyKperversity}.(ii)). 
By using \propref{prop:homotopyxR} and \corref{cor:kingargument}, the proof reduces to the following claim:
\begin{equation}\label{equa:exceptionalfinal}
H_{\star}(\crG_{\ov{p}}(\rc\S^{b-1},\rc \cI_{\S^{b-1}}))\cong H_{\star}(\crG_{\ov p}(\R^b,\I_{\R^b})).
\end{equation}
The cone formula \eqref{equa:cone} gives
\begin{equation}\label{equa:conesphere}
\pi_{k}^{\ov p}(\rc \S^{b-1},\rc\I_{\S^{b-1}})=\left\{
\begin{array}{cl}
\pi_{k}(\S^{b-1})&\text{if }  k\leq D\ov p(\tv),\\
0&\text{if }k>D\ov p(\tv).
\end{array}
\right.
\end{equation}
The hypothesis $\pi_{1}^{\ov p}(\rc \S^{b-1},\rc\I_{\S^{b-1}})=0$ implies that
the homological sphere $\S^{b-1}$ is a homotopy sphere $S^{b-1}$ if $1\leq  D\ov p(\tv)$.
From \eqref{eq:regS} applied to $D\ov p$, we have $D\ov p(\tv)\leq b-2$
and the equality \eqref{equa:conesphere} implies
$\pi_{k}^{\ov p}(\rc \S^{b-1},\rc\I_{\S^{b-1}})=0$
for any $k$,  which implies \eqref{equa:exceptionalfinal}.
\end{proof}

 \begin{proof}[Proof of \thmref{thm:coarseningnoexcep}]
 We apply \lemref{lem:quillenfini}.
 The isomorphisms on $\pi_{0}$ and $\pi_{1}$ have been 
 established in  Propositions \ref{lemapi1} and \ref{prop:noexceppi1}.
The isomorphisms in homology are in \thmref{thm:Coars}.
 \end{proof}
 
 \begin{proof}[Proof of \thmref{thm:coarseningwithexcep}]
It is a proof similar to that of \thmref{thm:coarseningnoexcep}. The only modifications are the replacement of 
\propref{prop:noexceppi1} by \propref{prop:pi1ThomM} and of \thmref{thm:Coars} by \thmref{thm:CoarsThom}.
 \end{proof}

\section{Refinement and GM-perversities}\label{sec:refinementexamples}

\begin{quote}
Let $(X,\cS,\cT)$ be a CS-coarsening. In this section, we use 
Theorems~\ref{thm:coarseningnoexcep} and \ref{thm:coarseningwithexcep}
to get isomorphisms on intersection homotopy groups, with a perversity $\ov q$ on $(X,\cT)$
and its pullback $\iota^{\star}\ov q$ on $(X,\cS)$.
We also  show how the GM-perversities fit in our setting.
\end{quote}

  \begin{theoremb}[\bf Invariance by refinement]\label{thm:raffinement}
 Let $(X,\cS,\cT,x_{0})$ be a pointed CS-coars\-ening  without exceptional strata, 
  and let $\ov q$ be a perversity on $(X,\cT)$ with   $\ov q \leq \ov t$, of pullback perversity $\iota^{\star}\ov q$. 
  Then,  the identity map induces an isomorphism
$$
\pi_\star^{\iota^\star\ov q}(X,\cS,x_0)  \cong \pi_\star^{\ov q}(X,\cT,x_0).
$$
  \end{theoremb}

\begin{proof}
From \lemref{Dback}, we know that $\iota^\star \ov q$ is a $K$-perversity and $\iota^\star\ov q\leq\ov t$.
So, the perversity $\ov p:=\iota^\star\ov q$ satisfies the hypotheses of \thmref{thm:coarseningnoexcep}.
The  result now follows from  $\iota_\star \ov p = \iota_\star \iota^\star \ov q =\ov q$,
as was noticed after \defref{def:pushback}. 
\end{proof}

With the same proof than in \thmref{thm:coarseningwithexcep}, we also get  
a refinement theorem in the case of a Thom-Mather space with some exceptional strata.

\begin{theoremb}[\bf Invariance by refinement]\label{thm:raffinementThom}
Let $(X,\cS)$ be a normal connected Thom Mather space   with a finite number of strata and $(X,\cS,\cT)$ be a 
 CS-coarsening   without exceptional strata of codimension 1.
Let $\ov q$ be a perversity on $(X,\cT)$ with   $\ov q \leq \ov t$,
of pullback  perversity $\iota^{\star}\ov q$, such that  
$\pi_{1}^{\iota^\star\ov q}(\rc L,\rc\cS,x_{0})=\{1\}$
for the links $L$ 
of any exceptional stratum  with $x_{0}$ regular.
Then, the identity map induces an isomorphism
$$\pi_{\star}^{\iota^\star\ov q}(X,\cS,x_{0})\cong \pi_{\star}^{\ov{q}}(X,\cT,x_{0}).$$
\end{theoremb}

If a stratum $S$ is exceptional, we have $\ov q=0$ in $W\menos S$, where $W$ is a small enough neighborhood of $S$.
Thus, $\iota^\star \ov q=0$ on $W\menos S$ and therefore on the link $L$. 
In short, the hypothesis on the links of exceptional strata can also be stated  as
$\pi_{1}^{\ov 0}(\rc L,\rc S,x_{0})=\{1\}$.

\medskip
Invariance of intersection homology was  established by Goresky and MacPherson  in the PL setting (\cite{GM1}) and
after  in the topological setting (\cite{GM2}). They use codimensional perversities with a growing property; we recall
their definition.

\begin{definition}\label{def:GMperversity}
A \emph{GM-perversity} is a  codimensional perversity $\ov p$ verifying $\ov p(0)=\ov p(1) =\ov p(2) = 0$ 
and, for any   $1\leq k$,
 \begin{equation}\label{eq:paso}
\ov p(k) \leq \ov p(k+1) \leq \ov p(k) +1.
\end{equation}
\end{definition}

We remind from \defref{def:perversity} that the top perversity~$\ov t$
is defined by $\ov{t}(k)=k-2$ if $k> 0$. Also, any codimensional perversity $\ov p$ verifies $\ov p(0)=0$,

\medskip
In \cite{CST9}, we have established the topological invariance of the intersection homotopy groups for 
GM-perversities and CS sets without exceptional strata. 
Below, we show  that \thmref{thm:coarseningnoexcep} extends this result to
 the perversities used by H. King in \cite{MR800845}.

\begin{proposition}\label{prop:GM}
Let $\ov p$ be a codimensional perversity verifying  $\ov p \leq \ov t$ and the growing property \eqref{eq:paso}.
Let $(X,\cS,\cT,x_{0})$ be a pointed  CS-coarsening without exceptional strata, then  
the identity map  induces  the isomorphism
$$
\iota_{\star}\colon \pi_{\star}^{\ov p}(X,\SS,x_0)  \cong \pi_{\star}^{ \ov p}(X,\TT,x_0).
$$
\end{proposition}

\begin{proof}
We also denote  $\ov p$ the perversity on $(X,\cS)$ induced by the codimensional perversity $\ov p$;
i.e., we set $\ov p(S)=\ov p(\codim\,S)$. We  claim that
$\ov p$ is a $K$-perversity for the CS-coarsening $(X,\cS,\cT)$. %
First, let $(S,Q)$ be a couple of strata of $\cS$ such that $S\preceq Q$ and $S^\iota=Q^\iota$.
For the property (K1) of \defref{def:Kper},  we have to check:
\begin{equation*}\label{eq:desigprob}
 \ov p(\codim\, Q)  \leq \ov p(\codim \,S)   \leq  \ov p(\codim \,Q)    + \ov t( \codim \,S) - \ov t( \codim \,Q).
\end{equation*} 
This is clear if $S, Q$ are regular strata. If they are both singular, this comes from
\eqref{eq:paso}.  
By hypothesis, there is no exceptional strata.

\smallskip
For the second property, we consider a couple of strata $(S,Q)$ of $\cS$ such that $\dim S=\dim Q$ and $S^\iota=T^\iota$. The axiom (K2) is trivially satisfied since we have
$$\ov p(S) = \ov p  (\codim \,S) = \ov p(\codim \,Q) = \ov p (Q).
$$

\smallskip
If $S\in \cS$ is a singular stratum, the inequalities 
$$
\ov p(S) = \ov p(\codim\, S) \leq  \ov t (\codim\, S) = \codim\, S -2 = \ov t(S).
$$
give $\ov p\leq \ov t$ on $(X,\cS)$.
Finally, let us  compute  $\iota_\star \ov p$. Let $T \in \cT$ and $S\in \cS$ be a source stratum of $T$,
then from \lemref{LemTec} i), we deduce
$$
\iota_\star \ov p(T ) =\ov p(S) = \ov p (\codim\, S)  = \ov p (\codim\, T) = \ov p(T).
$$
The hypotheses being verified, we can apply \thmref{thm:coarseningnoexcep} and get the result.
\end{proof}

 Notice that \propref{prop:GM}, Theorems  \ref{thm:Coars}   and \ref{thm:CoarsThom} also imply 
 a topological invariance of the homology of the involved Gajer spaces.
 

\end{document}